\documentclass[pdflatex,sn-mathphys-num]{sn-jnl}% Math and Physical Sciences Numbered Reference Style
\usepackage{graphicx}%
\usepackage{multirow}%
\usepackage{amsmath,amssymb,amsfonts}%
\usepackage{amsthm}%
\usepackage{mathrsfs}%
\usepackage[title]{appendix}%
\usepackage{xcolor}%
\usepackage{textcomp}%
\usepackage{manyfoot}%
\usepackage{booktabs}%
\usepackage{algorithm}%
\usepackage{algorithmicx}%
\usepackage{algpseudocode}%
\usepackage{listings}%
\theoremstyle{thmstyleone}%
\theoremstyle{thmstyletwo}%

\theoremstyle{thmstylethree}%
\newtheorem{assumption}{Assumption}[section]%
\newtheorem{theorem}{Theorem}[section]%
\newtheorem{lemma}[theorem]{Lemma}%
\newtheorem{definition}[theorem]{Definition}%
\newtheorem{proposition}[theorem]{Proposition}% 
\newtheorem{remark}{Remark}[section]%

\begin{document}

\title[Article Title]{Kernel-Based Heterogeneous Mean Field Games on a Continuum Type Space}


\author[1]{\fnm{Guangchen} \sur{Wang}}\email{wguangchen@sdu.edu.cn}

\author*[1]{\fnm{Luebo} \sur{Xia}}\email{202414833@mail.sdu.edu.cn}

\affil*[1]{\orgdiv{School of Control Science and Engineering}, \orgname{Shandong University}, \orgaddress{\street{Jingshi Rd.}, \city{Jinan}, \postcode{250061},  \country{PR China}}}


\abstract{This paper studies a class of heterogeneous mean field games on a continuum type space, in which interactions are generated by a Hilbert-Schmidt kernel operator acting on the type-indexed mean state profile. The control set is allowed to be a general closed convex subset of a Euclidean space, and the optimal feedback is characterized through the unique maximizer of the Hamiltonian. For a prescribed mean-field path, the stochastic maximum principle leads to a family of fully coupled forward-backward stochastic differential equations (FBSDEs) indexed by the agent type. Under Lipschitz conditions and a monotonicity assumption on the reduced coefficients, the frozen FBSDE is globally well posed. We then formulate the consistency problem as a fixed-point problem for the kernel mean-field path. When the operator norm of the interaction kernel is sufficiently small, the resulting map has an invariant ball and is continuous and compact, so that Schauder's theorem yields existence of a mean field Nash equilibrium. Uniqueness follows under an aggregate monotonicity condition adapted to the kernel coupling. A linear-quadratic example illustrates how finite-rank interactions and heterogeneous risk preferences produce coupled collective modes, and yields a spectral smallness condition for existence and an explicit sufficient condition for uniqueness.}




\maketitle

\section{Introduction}\label{sec1}

Mean field game (MFG) theory provides a tractable framework for the analysis of stochastic differential games with a large population of weakly interacting agents. In the limiting model, the influence of an individual player on the population becomes negligible, while the aggregate behavior of the population enters the optimization problem of each representative player. Since the seminal works \cite{bib2,bib3}, the theory has developed along both analytic and probabilistic directions; see, for example, \cite{bib1,bib4,bib4_1}. In the classical formulation, agents are statistically homogeneous and interact through a common population distribution.

A substantial part of the recent literature has been devoted to models in which this homogeneity is relaxed. Heterogeneity may arise through different coefficients, constraints, labels, subpopulations, or network positions; see, for example, \cite{bib5,bib6,cainesHuang2021graphon,aurellCarmonaLauriere2022,lackerSoret2023}. Graphon and kernel-based models provide a particularly natural framework for describing heterogeneous interactions indexed by a continuum of agent types. In such models, the influence exerted by agents of type $\theta'$ on agents of type $\theta$ is weighted by a kernel depending on the pair $(\theta,\theta')$. Graphon mean field games and their associated equations were developed in \cite{cainesHuang2021graphon}, while stochastic graphon games with a continuum of players were investigated, in particular, in \cite{aurellCarmonaLauriere2022}. A different label-state formulation, which embeds the player label into the state and permits considerably more general graphon interactions, was proposed in \cite{lackerSoret2023}.

The present paper studies a probabilistic kernel-based MFG on a general Polish probability type space $(\Theta,\mathcal B(\Theta),\lambda)$. For a type-indexed measure flow
$
\mu=\{\mu_t^\theta\}_{t,\theta},
$
we consider the interaction
$$
m_t^K[\mu](\theta) = \int_\Theta K(\theta,\theta') \left( \int_{\mathbb R^k}x\,\mu_t^{\theta'}(\mathrm dx) \right) \lambda(\mathrm d\theta').
$$
Thus the environment faced by type $\theta$ is obtained by applying an integral operator to the first-moment profile of the population. The kernel is allowed to be matrix-valued and need not be symmetric. Our main structural assumption is that it belongs to $L^2(\Theta\times\Theta)$, so that the associated integral operator on $L^2(\Theta;\mathbb R^k)$ is Hilbert-Schmidt and hence compact.

Our formulation differs from the linear-quadratic continuum-player models in that the controlled dynamics and reward functions may be nonlinear. The control set $U$ is an arbitrary nonempty closed convex subset of a Euclidean space. We assume that the Hamiltonian admits a unique maximizer $\hat u(t,x,p,m,\theta)$, with linear growth, and formulate the main regularity and monotonicity assumptions directly in terms of the resulting reduced coefficients. This makes it possible to separate the two roles played by the stochastic control problem: first, the stochastic maximum principle identifies the candidate optimal feedback and leads to a reduced forward-backward stochastic differential equation (FBSDE); second, suitable concavity inherited from the reduced monotonicity conditions verifies that this feedback is indeed optimal.

For a prescribed measure flow, the kernel mean field becomes an exogenous deterministic path. The reduced optimality system is then a family of fully coupled FBSDEs indexed by $\theta$. We impose a sign-adjusted Peng-Wu monotonicity condition on the reduced forward and backward coefficients and on the terminal condition. The classical continuation method for monotone FBSDEs \cite{pengwu1999} then gives global well-posedness of the frozen system. In addition to providing the single-valued response needed below, the same structural condition implies strong concavity of the maximized Hamiltonian in the state variable and strong concavity of the terminal reward, which permits a direct verification of optimality.

The consistency problem has two equivalent descriptions. In the sense of state distributions, the frozen FBSDE defines the map $\Phi(\mu)_t^\theta = \mathcal L(X_t^{\mu,\theta})$. For the existence proof, the fixed-point argument is carried out via the kernel mean-field path. Given a prescribed path $m\in C([0,T];L^2(\Theta;\mathbb R^k))$, we solve the corresponding frozen FBSDE and apply the kernel operator to the resulting mean state profile. The stability estimate for monotone FBSDEs gives continuity of this response, while compactness follows from the Hilbert-Schmidt property of the kernel operator together with a uniform time-modulus estimate. If the operator norm $\kappa_K$ is sufficiently small relative to the a priori constants of the frozen FBSDE, an invariant closed ball can be constructed. Schauder's fixed-point theorem then yields a self-consistent mean-field path, from which a fixed point of $\Phi$ and hence a mean field Nash equilibrium are recovered.

The smallness assumption used in the existence argument should be distinguished from a contraction condition. We do not require $\Phi$ to be a contraction. In fact, the stability estimate only gives $d(\Phi(\mu),\Phi(\nu)) \le C_\Phi\kappa_K d(\mu,\nu)$, while the Schauder argument merely requires enough smallness to preserve an invariant ball. This distinction also becomes important for uniqueness. Rather than imposing the condition $C_\Phi\kappa_K<1$, we introduce an aggregate monotonicity condition adapted to the kernel coupling. It compares two complete type-indexed state-adjoint profiles after their corresponding kernel mean fields have been generated endogenously. An energy identity for the difference of two equilibrium FBSDEs then yields uniqueness of both the equilibrium measure flow and the associated optimal control.

The main contributions of the paper may therefore be summarized as follows:
\begin{itemize}

\item
We formulate a nonlinear heterogeneous MFG on a continuum type space with a matrix-valued Hilbert-Schmidt interaction kernel and a general closed convex control set. The stochastic maximum principle reduces each frozen individual optimization problem to a type-indexed fully coupled FBSDE. Under a Peng-Wu type monotonicity condition, this system is globally well posed, and its solution generates the unique optimal response.

\item
We establish existence of a mean field Nash equilibrium through a Schauder fixed-point argument on the Banach space of kernel mean-field paths. The proof combines a priori estimates for monotone FBSDEs, stability with respect to the prescribed mean-field path, compactness of the kernel operator, and an invariant-ball estimate. The required weak-interaction condition is expressed in terms of the operator norm $\kappa_K$, rather than a pointwise bound on the kernel.

\item
We prove uniqueness under an aggregate monotonicity condition which incorporates the kernel coupling directly. This condition is distinct from the smallness condition used for existence and can yield uniqueness outside the contraction regime. A readily verifiable sufficient condition is obtained by comparing the intrinsic monotonicity constants of the frozen FBSDE with the effective interaction strength $\kappa_K$.

\item
We study an explicit linear-quadratic model and derive concrete existence and uniqueness conditions in terms of the model matrices and $\kappa_K$. For finite-rank kernels, the directly interacting component is generated by finitely many kernel modes, while heterogeneous preferences induce additional coupling between these modes through a multiplication operator on the type space. The corresponding eigenvalues quantify both the strength of the interaction modes and the conditions ensuring well-posedness.

\end{itemize}

The remainder of the paper is organized as follows. Section \ref{sec2} introduces the functional setting, the kernel-based heterogeneous mean field game, the structural assumptions, and the reduced FBSDE obtained from the stochastic maximum principle. Section \ref{sec3} establishes well-posedness and stability of the frozen FBSDE, verifies optimality of the resulting feedback controls, proves equilibrium existence by a Schauder fixed-point argument, and establishes equilibrium uniqueness under aggregate kernel monotonicity. Section \ref{sec4} develops a linear-quadratic example and analyzes finite-rank kernels through their spectral decomposition. Section \ref{sec5} concludes with several directions for further research.

\section{Preliminaries}\label{sec2}

\subsection{Notation and Functional Setting}
\label{subsec2.1}

Let $(\Omega,\mathcal F,\mathbb F,\mathbb P)$ be a filtered probability space satisfying the usual conditions, where $\mathbb F=\{\mathcal F_t\}_{t\in[0,T]}$. We assume that it supports a $k$-dimensional Brownian motion $W$ and that $\mathcal F_t = \mathcal F_0\vee\sigma(W_s:0\le s\le t)$ up to the usual augmentation, where $W$ is independent of $\mathcal F_0$.

The type space is denoted by a Polish probability space $(\Theta,\mathcal B(\Theta),\lambda)$. The element $\theta\in\Theta$ represents the type of an individual agent, while $\lambda$ describes the distribution of types. For measurable functions $h:\Theta\to\mathbb R^k$, we write
$$
\|h\|_{L^2_\lambda}^2 = \int_\Theta |h(\theta)|^2 \lambda(\mathrm d\theta).
$$

For $p\ge1$, let $\mathcal P_p(\mathbb R^k)$ denote the set of Borel probability measures on $\mathbb R^k$ with finite $p$-th moment. The $2$-Wasserstein distance on $\mathcal P_2(\mathbb R^k)$ is defined by
$$
\mathcal W_2(\mu,\nu) = \inf_{\pi\in\Pi(\mu,\nu)} \left( \int_{\mathbb R^k\times\mathbb R^k} |x-y|^2 \pi(\mathrm dx,\mathrm dy) \right)^{1/2},
$$
where $\Pi(\mu,\nu)$ denotes the set of all couplings between $\mu$ and $\nu$.

We shall use measurable type-indexed families of probability measures. Define
$$
\mathscr P_{2,\lambda} := \left\{ \boldsymbol\mu=\{\mu^\theta\}_{\theta\in\Theta}:
\begin{array}{l}
\theta\mapsto\mu^\theta \text{ is measurable},
\\[1mm]
\displaystyle
\int_\Theta \int_{\mathbb R^k} |x|^2\mu^\theta(\mathrm dx) \lambda(\mathrm d\theta)<\infty
\end{array}
\right\},
$$
where families equal $\lambda$-a.e. are identified. We equip this space with
$$
\mathbf W_{2,\lambda} (\boldsymbol\mu,\boldsymbol\nu) = \left( \int_\Theta \mathcal W_2(\mu^\theta,\nu^\theta)^2 \lambda(\mathrm d\theta) \right)^{1/2}.
$$

The space of admissible measure flows is $\mathcal M = C([0,T];\mathscr P_{2,\lambda})$, where continuity is understood with respect to $\mathbf W_{2,\lambda}$. For $\mu,\nu\in\mathcal M$, define
$$
d(\mu,\nu) = \sup_{t\in[0,T]} \left( \int_\Theta \mathcal W_2 (\mu_t^\theta,\nu_t^\theta)^2 \lambda(\mathrm d\theta) \right)^{1/2}.
$$

The kernel interaction is generated by a measurable matrix-valued kernel $K\in L^2(\Theta\times\Theta;\mathbb R^{k\times k})$. The associated integral operator
$$
\mathcal K:
L^2(\Theta;\mathbb R^k)
\to
L^2(\Theta;\mathbb R^k)
$$
is defined by
$$
(\mathcal Kh)(\theta)
=
\int_\Theta
K(\theta,\theta')
h(\theta')
\lambda(\mathrm d\theta').
$$
We denote its operator norm by
$$
\kappa_K
:=
\big\|\mathcal K\big\|_{\mathcal L(L^2(\Theta;\mathbb R^k))}.
$$
Since $K$ is square-integrable, $\mathcal K$ is Hilbert-Schmidt and hence compact.

For matrices, $\|\cdot\|$ denotes the operator norm, while $\|\cdot\|_{\mathrm F}$ denotes the Frobenius norm.

For a jointly measurable type-indexed family of continuous adapted processes $X = \{X^\theta\}_{\theta\in\Theta}$, define
$$
\|X\|_{\mathbb S_\Theta^2}^2
=
\int_\Theta
\mathbb E
\left[
\sup_{t\in[0,T]}
|X_t^\theta|^2
\right]
\lambda(\mathrm d\theta).
$$
We write $X\in\mathbb S_\Theta^2$ whenever this quantity is finite.

For a jointly measurable type-indexed family of progressively measurable processes $Z = \{Z^\theta\}_{\theta\in\Theta}$, define
$$
\|Z\|_{\mathbb H_\Theta^2}^2
=
\int_\Theta
\mathbb E
\left[
\int_0^T
|Z_t^\theta|^2\mathrm dt
\right]
\lambda(\mathrm d\theta),
$$
and write $Z\in\mathbb H_\Theta^2$ whenever this quantity is finite.

For a jointly measurable type-indexed square-integrable random variable $X=\{X^\theta\}_{\theta\in\Theta}$, define
$$
\mathfrak m[X](\theta)
:=
\mathcal K
\bigl(
\theta'\mapsto
\mathbb E[X^{\theta'}]
\bigr)(\theta).
$$

The control space is
$$
\mathcal U
=
\left\{
u:[0,T]\times\Omega\to U:
\begin{array}{l}
u\text{ is }\mathbb F\text{-progressively measurable},
\\[1mm]
\displaystyle
\mathbb E\int_0^T|u_t|^2\mathrm dt<\infty
\end{array}
\right\},
$$
where $U\subset\mathbb R^m$ is a nonempty closed convex set.

Throughout the paper, generic positive constants are denoted by $C$, possibly with subscripts. Their values may change from line to line.

\subsection{Kernel-based Heterogeneous Mean Field Game}
\label{subsec2.2}

For a measure flow $\mu\in\mathcal M$, define its first-moment profile by
\begin{equation}
\label{1st_moment_profile}
a_t^\mu(\theta)
:=
\int_{\mathbb R^k}
x\,\mu_t^\theta(\mathrm dx),
\end{equation}
and define the corresponding kernel mean field by
$$
m_t^K[\mu](\theta)
:=
\int_\Theta
K(\theta,\theta')
a_t^\mu(\theta')
\lambda(\mathrm d\theta').
$$
Equivalently,
$$
m_t^K[\mu]
=
\mathcal K a_t^\mu.
$$

For a prescribed measure flow $\mu\in\mathcal M$ and a fixed type $\theta$, the process
$$
m^K[\mu](\theta)
=
\{m_t^K[\mu](\theta)\}_{t\in[0,T]}
$$
is deterministic and is treated as an exogenous environment in the individual stochastic control problem.

For a representative agent of type $\theta$ and a control
$u^\theta\in\mathcal U$, the state process satisfies
$$
\mathrm dX_t^\theta
=
b\bigl(
t,X_t^\theta,u_t^\theta,
m_t^K[\mu](\theta),\theta
\bigr)\mathrm dt
+
\sigma\,\mathrm dW_t,
\qquad
X_0^\theta=\xi^\theta,
$$
where
$$
b:
[0,T]\times\mathbb R^k\times U\times\mathbb R^k\times\Theta
\longrightarrow
\mathbb R^k
$$
and $\sigma\in\mathbb R^{k\times k}$ is a constant nondegenerate matrix.

The reward functional is
$$
J^\theta(u^\theta;m^K[\mu])
:=
\mathbb E
\left[
\int_0^T
f\bigl(
t,X_t^\theta,u_t^\theta,
m_t^K[\mu](\theta),\theta
\bigr)
\mathrm dt
+
g\bigl(
X_T^\theta,m_T^K[\mu](\theta),\theta
\bigr)
\right],
$$
where
$$
f:
[0,T]\times\mathbb R^k\times U\times\mathbb R^k\times\Theta
\longrightarrow
\mathbb R
$$
and
$$
g:
\mathbb R^k\times\mathbb R^k\times\Theta
\longrightarrow
\mathbb R.
$$
For
$$
(t,x,u,p,m,\theta)
\in
[0,T]\times\mathbb R^k\times U
\times\mathbb R^k\times\mathbb R^k\times\Theta,
$$
define the Hamiltonian
$$
H(t,x,u,p;\theta,m)
:=
\langle b(t,x,u,m,\theta),p\rangle
+
f(t,x,u,m,\theta).
$$
The maximizer of the Hamiltonian
$$
\hat u(t,x,p,m,\theta)
:=
\arg\max_{u\in U}
H(t,x,u,p;\theta,m)
$$
is well-defined and unique under strict concavity and coercivity conditions. (See Assumption \ref{asmp004}(a) and (b) and \cite{bonnansShapiro2000} for details.)
For later use, define the maximized Hamiltonian by
$$
\widehat H(t,x,p;\theta,m)
:=
H\bigl(
t,x,\hat u(t,x,p,m,\theta),p;\theta,m
\bigr).
$$
We also introduce the reduced coefficients
$$
B(t,x,p,m,\theta)
:=
b\bigl(
t,x,\hat u(t,x,p,m,\theta),m,\theta
\bigr),
$$
$$
F(t,x,p,m,\theta)
:=
\partial_xH
\bigl(
t,x,\hat u(t,x,p,m,\theta),p;\theta,m
\bigr),
$$
and
$$
\Gamma(x,m,\theta)
:=
\partial_xg(x,m,\theta).
$$

It is useful to state explicitly the equilibrium notion used below.

\begin{definition}
\label{def:MFG_equilibrium}
A mean field Nash equilibrium is a pair
$$
\left(
\{u^{\ast,\theta}\}_{\theta\in\Theta},
\mu^\ast
\right),
\qquad
\mu^\ast\in\mathcal M,
$$
such that, for $\lambda$-a.e. $\theta\in\Theta$:

\begin{enumerate}
\item[(i)]
$u^{\ast,\theta}\in\mathcal U$ is optimal against the environment generated by $\mu^\ast$, that is,
$$
J^\theta
(u^{\ast,\theta};m^K[\mu^\ast])
=
\sup_{u^\theta\in\mathcal U}
J^\theta
(u^\theta;m^K[\mu^\ast]);
$$

\item[(ii)]
if $X^{\ast,\theta}$ denotes the state process corresponding to
$u^{\ast,\theta}$ and $m^K[\mu^\ast]$, then
$$
\mu_t^{\ast,\theta}
=
\mathcal L(X_t^{\ast,\theta}),
\qquad
t\in[0,T].
$$
\end{enumerate}
\end{definition}

We now collect the standing assumptions.

\begin{assumption}
\label{asmp001}
The family
$$
(\theta,\omega)
\longmapsto
\xi^\theta(\omega)
$$
is jointly measurable, $\xi^\theta$ is $\mathcal F_0$-measurable for $\lambda$-a.e. $\theta$, and
$$
\int_\Theta
\mathbb E|\xi^\theta|^2
\lambda(\mathrm d\theta)
<\infty.
$$
\end{assumption}

\begin{assumption}
\label{asmp002}
The primitive coefficients satisfy the following conditions.

\item[(a)]
The functions $b$ and $f$ are jointly measurable in $(t,\theta)$ and continuous in $(x,u,m)$, while $g$ is measurable in $\theta$ and continuous in $(x,m)$.

\item[(b)]
There exists $L_b>0$, independent of $(t,\theta)$, such that
$$
|b(t,x,u,m,\theta)
-
b(t,x',u,m,\theta)|
\le
L_b|x-x'|,
$$
and there exists $C>0$ such that
$$
|b(t,x,u,m,\theta)|
\le
C(1+|x|+|u|+|m|).
$$

\item[(c)]
The functions $b$ and $f$ are continuously differentiable with respect to $x$, and $g$ is continuously differentiable with respect to $x$. Their derivatives are continuous in the remaining Euclidean variables and satisfy
$$
|\partial_xb(t,x,u,m,\theta)|
\le C,
$$
$$
|\partial_xf(t,x,u,m,\theta)|
\le
C(1+|x|+|u|+|m|),
$$
and
$$
|\partial_xg(x,m,\theta)|
\le
C(1+|x|+|m|).
$$

\item[(d)]
The rewards have at most quadratic growth:
$$
|f(t,x,u,m,\theta)|
+
|g(x,m,\theta)|
\le
C
\bigl(
1+|x|^2+|u|^2+|m|^2
\bigr).
$$
\end{assumption}

\begin{assumption}
\label{asmp003}
Write
$$
\kappa_K
:=
\|\mathcal K\|_{\mathcal L(L^2(\Theta;\mathbb R^k))}.
$$
We assume
\begin{equation}
\label{eq:kernel_smallness}
\kappa_K<\bar\kappa,
\end{equation}
where the quantitative threshold $\bar\kappa>0$ is chosen in Section \ref{sec3} so that the invariant-ball condition for the consistency map is satisfied.
\end{assumption}

\begin{assumption}
\label{asmp004}
The Hamiltonian satisfies the following conditions.

\item[(a)]
For every $(t,x,p,m,\theta)$, the mapping
$$
u
\longmapsto
H(t,x,u,p;\theta,m)
$$
is continuous and strictly concave on $U$.

\item[(b)]
For every fixed $(t,x,p,m,\theta)$,
$$
H(t,x,u,p;\theta,m)
\longrightarrow
-\infty
\qquad
\text{as }|u|\to\infty,\quad u\in U.
$$

\item[(c)]
The resulting unique maximizer
$$
\hat u(t,x,p,m,\theta)
:=
\arg\max_{u\in U}
H(t,x,u,p;\theta,m)
$$
is jointly measurable in all its variables.
\end{assumption}

\begin{assumption}
\label{asmp005}
The maximizer and maximized Hamiltonian satisfy:

\item[(a)]
There exists $C_u>0$, independent of $(t,\theta)$, such that
$$
|\hat u(t,x,p,m,\theta)|
\le
C_u
\bigl(
1+|x|+|p|+|m|
\bigr).
$$

\item[(b)]
For every $(t,m,\theta)$, the mapping
$$
(x,p)
\longmapsto
\widehat H(t,x,p;\theta,m)
$$
is continuously differentiable and satisfies the envelope identities
$$
\partial_x\widehat H(t,x,p;\theta,m)
=
\partial_xH
\bigl(
t,x,\hat u(t,x,p,m,\theta),p;\theta,m
\bigr)
$$
and
$$
\partial_p\widehat H(t,x,p;\theta,m)
=
b\bigl(
t,x,\hat u(t,x,p,m,\theta),m,\theta
\bigr).
$$
\end{assumption}

\begin{assumption}
\label{asmp006}
The reduced coefficients satisfy the following conditions.

\item[(a)]
The functions $B$ and $F$ are jointly measurable in $(t,\theta)$, and $\Gamma$ is measurable in $\theta$. There exist constants
$$
L_{xp},L_m,L_\Gamma,L_{\Gamma,m},C_0>0
$$
such that
$$
\begin{aligned}
&
|B(t,x,p,m,\theta)-B(t,x',p',m',\theta)|
\\
&\quad+
|F(t,x,p,m,\theta)-F(t,x',p',m',\theta)|
\\
&\le
L_{xp}
\bigl(
|x-x'|+|p-p'|
\bigr)
+
L_m|m-m'|,
\end{aligned}
$$
and
$$
|\Gamma(x,m,\theta)-\Gamma(x',m',\theta)|
\le
L_\Gamma|x-x'|
+
L_{\Gamma,m}|m-m'|.
$$
Moreover,
$$
\sup_{(t,\theta)}
\Bigl(
|B(t,0,0,0,\theta)|
+
|F(t,0,0,0,\theta)|
\Bigr)
+
\sup_\theta
|\Gamma(0,0,\theta)|
\le C_0.
$$

\item[(b)]
There exist constants $\beta,\nu>0$, independent of $(t,m,\theta)$, such that for all $x,x',p,p'\in\mathbb R^k$,
\begin{equation}
\label{eq:PW_monotonicity}
\begin{aligned}
&
\left\langle
B(t,x,p,m,\theta)-B(t,x',p',m,\theta),
p-p'
\right\rangle
\\
&\quad-
\left\langle
F(t,x,p,m,\theta)-F(t,x',p',m,\theta),
x-x'
\right\rangle
\\
&\ge
\beta|x-x'|^2,
\end{aligned}
\end{equation}
and
\begin{equation}
\label{eq:PW_terminal}
\left\langle
\Gamma(x,m,\theta)-\Gamma(x',m,\theta),
x-x'
\right\rangle
\le
-\nu|x-x'|^2.
\end{equation}
\end{assumption}

The following assumption will only be needed for uniqueness of the self-consistent equilibrium.

\begin{assumption}
\label{asmp007}
There exist constants $c_{\mathrm M}>0$ and $c_{\mathrm T}>0$ such that:

\item[(a)]
For every $t\in[0,T]$ and all jointly measurable square-integrable families
$$
X=\{X^\theta\}_{\theta\in\Theta},
\qquad
P=\{P^\theta\}_{\theta\in\Theta},
$$
and
$$
\widetilde X=\{\widetilde X^\theta\}_{\theta\in\Theta},
\qquad
\widetilde P=\{\widetilde P^\theta\}_{\theta\in\Theta},
$$
one has
\begin{equation}
\label{eq:kernel_monotonicity}
\begin{aligned}
&
\int_\Theta
\mathbb E
\Big[
\big\langle
B(t,X^\theta,P^\theta,\mathfrak m[X](\theta),\theta)
\\
&\qquad\qquad
-
B(t,\widetilde X^\theta,\widetilde P^\theta,
\mathfrak m[\widetilde X](\theta),\theta),
P^\theta-\widetilde P^\theta
\big\rangle
\\
&\qquad
-
\big\langle
F(t,X^\theta,P^\theta,\mathfrak m[X](\theta),\theta)
\\
&\qquad\qquad
-
F(t,\widetilde X^\theta,\widetilde P^\theta,
\mathfrak m[\widetilde X](\theta),\theta),
X^\theta-\widetilde X^\theta
\big\rangle
\Big]
\lambda(\mathrm d\theta)
\\
&\ge
c_{\mathrm M}
\int_\Theta
\mathbb E
|X^\theta-\widetilde X^\theta|^2
\lambda(\mathrm d\theta).
\end{aligned}
\end{equation}

\item[(b)]
For every pair of jointly measurable square-integrable families
$X$ and $\widetilde X$,
\begin{equation}
\label{eq:kernel_terminal_monotonicity}
\begin{aligned}
&
\int_\Theta
\mathbb E
\Big[
\big\langle
\Gamma(X^\theta,\mathfrak m[X](\theta),\theta)
\\
&\qquad\qquad
-
\Gamma(\widetilde X^\theta,
\mathfrak m[\widetilde X](\theta),\theta),
X^\theta-\widetilde X^\theta
\big\rangle
\Big]
\lambda(\mathrm d\theta)
\\
&\le
-c_{\mathrm T}
\int_\Theta
\mathbb E
|X^\theta-\widetilde X^\theta|^2
\lambda(\mathrm d\theta).
\end{aligned}
\end{equation}
\end{assumption}

\begin{remark}
Since $K\in L^2(\Theta\times\Theta)$, the operator $\mathcal K$ is Hilbert-Schmidt and therefore compact. In particular,
$$
\kappa_K
\le
\|K\|_{L^2(\Theta\times\Theta)}.
$$
Hence the directly verifiable condition
$$
\|K\|_{L^2(\Theta\times\Theta)}
<
\bar\kappa
$$
is sufficient for \eqref{eq:kernel_smallness}. No symmetry assumption on $K$ is imposed in the general theory. If $\mathcal K$ is self-adjoint, then compactness gives
$$
\kappa_K
=
\sup_j|\kappa_j|,
$$
where $\{\kappa_j\}$ are its eigenvalues.
\end{remark}

\begin{remark}
\label{rem:reduced_concavity}
Assumption \ref{asmp006}(b) is a sign-adjusted version of the monotonicity condition used in the continuation method for fully coupled FBSDEs introduced by Peng and Wu \cite{pengwu1999}. In the present model the forward diffusion coefficient is constant and the reduced coefficients are independent of the martingale integrand, so the relevant condition reduces to \eqref{eq:PW_monotonicity}--\eqref{eq:PW_terminal}.

The same assumption also provides the concavity needed in the verification argument. Indeed, by the envelope identity in Assumption \ref{asmp005}, taking $p'=p$ in \eqref{eq:PW_monotonicity} gives
$$
\left\langle
\partial_x\widehat H(t,x,p;\theta,m)
-
\partial_x\widehat H(t,x',p;\theta,m),
x-x'
\right\rangle
\le
-\beta|x-x'|^2.
$$
Hence
$$
\begin{aligned}
\widehat H(t,y,p;\theta,m)
\le{}&
\widehat H(t,x,p;\theta,m)
\\
&+
\left\langle
\partial_x\widehat H(t,x,p;\theta,m),
y-x
\right\rangle
-
\frac{\beta}{2}|y-x|^2.
\end{aligned}
$$
Similarly, since $\Gamma=\partial_xg$, condition \eqref{eq:PW_terminal} gives
$$
g(y,m,\theta)
\le
g(x,m,\theta)
+
\langle
\Gamma(x,m,\theta),y-x
\rangle
-
\frac{\nu}{2}|y-x|^2.
$$
Thus both
$
\widehat H(t,\cdot,p;\theta,m)
$
and
$
g(\cdot,m,\theta)
$
are strongly concave in the state variable.
\end{remark}

\begin{remark}
\label{rem:kernel_monotonicity_verification}
Assumption \ref{asmp007} is an aggregate monotonicity condition adapted to the endogenous kernel interaction. It is imposed only on mean fields generated by state profiles and is therefore different from requiring a pointwise monotonicity inequality for arbitrary external values of $m$.

A convenient sufficient condition can be obtained from Assumption \ref{asmp006}. Suppose, in addition, that there exists $\beta_p>0$ such that
$$
\begin{aligned}
&
\left\langle
B(t,x,p,m,\theta)-B(t,x',p',m,\theta),
p-p'
\right\rangle
\\
&\quad-
\left\langle
F(t,x,p,m,\theta)-F(t,x',p',m,\theta),
x-x'
\right\rangle
\\
&\ge
\beta|x-x'|^2
+
\beta_p|p-p'|^2.
\end{aligned}
$$
Since
$$
\|\mathfrak m[X]-\mathfrak m[\widetilde X]\|_{L^2_\lambda}
\le
\kappa_K
\left(
\int_\Theta
\mathbb E
|X^\theta-\widetilde X^\theta|^2
\lambda(\mathrm d\theta)
\right)^{1/2},
$$
the Lipschitz bounds in Assumption \ref{asmp006}(a) imply
\eqref{eq:kernel_monotonicity} whenever
$$
\beta
-
L_m\kappa_K
-
\frac{L_m^2\kappa_K^2}{4\beta_p}
>0.
$$
Likewise,
\eqref{eq:kernel_terminal_monotonicity} follows from
$$
\nu
-
L_{\Gamma,m}\kappa_K
>0.
$$
Thus the uniqueness condition can be verified by comparing the intrinsic monotonicity of the frozen state-adjoint system with the effective strength of the kernel interaction.
\end{remark}

\subsection{Stochastic Maximum Principle and the Reduced FBSDE}
\label{subsec2.3}

Fix a measure flow $\mu\in\mathcal M$ and a type $\theta\in\Theta$, and write
$$
m_t^{\mu,\theta}
:=
m_t^K[\mu](\theta).
$$
The process $m^{\mu,\theta}$ is deterministic and is regarded as an exogenous input to the individual control problem
$$
\sup_{u^\theta\in\mathcal U}
J^\theta(u^\theta;m^K[\mu]),
$$
subject to
$$
\mathrm dX_t^\theta
=
b(t,X_t^\theta,u_t^\theta,m_t^{\mu,\theta},\theta)\mathrm dt
+
\sigma\mathrm dW_t.
$$

Suppose that the above control problem admits an optimal pair $(X^\theta,u^\theta)$. Under the standing regularity assumptions, the stochastic maximum principle yields an adjoint pair $(p^\theta,q^\theta)$ satisfying
$$
\mathrm dp_t^\theta
=
-\partial_xH
(t,X_t^\theta,u_t^\theta,p_t^\theta;
\theta,m_t^{\mu,\theta})
\mathrm dt
+
q_t^\theta\mathrm dW_t,
$$
with terminal condition
$$
p_T^\theta
=
\partial_xg
(X_T^\theta,m_T^{\mu,\theta},\theta).
$$
Moreover, the maximum condition gives
$$
u_t^\theta
\in
\arg\max_{u\in U}
H(t,X_t^\theta,u,p_t^\theta;
\theta,m_t^{\mu,\theta}),
\qquad
\mathrm dt\otimes\mathrm d\mathbb P\text{-a.e.}
$$
By Assumption \ref{asmp004}, the maximizer is uniquely defined. Denote 
$$
\hat u(t,x,p,m,\theta)
:=
\arg\max_{u\in U}
H(t,x,u,p;\theta,m).
$$
Hence every optimal control necessarily satisfies
$$
u_t^\theta
=
\hat u
(t,X_t^\theta,p_t^\theta,m_t^{\mu,\theta},\theta).
$$

Consequently, it follows that every optimal pair, together with its adjoint processes, necessarily gives a solution of the reduced FBSDE
\begin{equation}
\left\{
\begin{aligned}
\mathrm dX_t^{\mu,\theta}
={}&
B(t,X_t^{\mu,\theta},p_t^{\mu,\theta},
m_t^{\mu,\theta},\theta)\mathrm dt
+
\sigma\mathrm dW_t,
\\
\mathrm dp_t^{\mu,\theta}
={}&
-F(t,X_t^{\mu,\theta},p_t^{\mu,\theta},
m_t^{\mu,\theta},\theta)\mathrm dt
+
q_t^{\mu,\theta}\mathrm dW_t,
\\
X_0^{\mu,\theta}
={}&
\xi^\theta,
\\
p_T^{\mu,\theta}
={}&
\Gamma(X_T^{\mu,\theta},m_T^{\mu,\theta},\theta).
\end{aligned}
\right.
\label{eq:reduced_FBSDE}
\end{equation}

At this stage, the stochastic maximum principle is used only as a necessary condition: a solution of \eqref{eq:reduced_FBSDE} is not, by itself, asserted to generate an optimal control. In the next section, we utilize the monotonicity condition of Peng-Wu introduced in Assumption \ref{asmp006}, under which the frozen FBSDE \eqref{eq:reduced_FBSDE} admits a unique adapted solution for every prescribed measure flow $\mu$. This unique solution will be used to define the consistency mapping on the space of measure flows. Its connection with optimal responses, and hence with mean field Nash equilibria, will be established separately.

\section{Existence and Uniqueness of Mean Field Nash Equilibria}
\label{sec3}

In this section, we establish the existence and uniqueness of a mean field Nash equilibrium through the reduced FBSDE derived in Subsection~\ref{subsec2.3}. The argument separates the individual optimization problem under a prescribed mean field from the population-level consistency problem.

We first consider a prescribed deterministic mean-field path. Under the monotonicity condition established in Assumption~\ref{asmp006}, the corresponding reduced FBSDE is globally well posed. The same monotonicity structure, together with the envelope identities in Assumption~\ref{asmp005}, implies strong concavity of the maximized Hamiltonian in the state variable and hence provides a direct verification argument. Consequently, the unique solution of the frozen FBSDE generates the unique optimal response to the prescribed mean field. This allows us to define a single-valued consistency map $\Phi$ by assigning to each prescribed measure flow the distribution flow of the optimally controlled state.

We then establish stability of the frozen FBSDE with respect to the prescribed mean field and, as a consequence, continuity of $\Phi$. For the existence argument, we formulate an equivalent consistency map $\mathcal T$ on $C([0,T];L^2(\Theta;\mathbb R^k))$ by applying the kernel operator to the mean state profile generated by an arbitrary prescribed mean-field path. The Hilbert-Schmidt property of the kernel yields the spatial compactness needed for this map, while the FBSDE estimates provide a uniform modulus of continuity in time. Under the weak-interaction condition in Assumption~\ref{asmp003}, an invariant closed ball can be constructed, and Schauder's fixed-point theorem then gives a self-consistent mean-field path and hence a mean field Nash equilibrium.

Finally, uniqueness is treated separately from existence. We impose the aggregate monotonicity condition in Assumption~\ref{asmp007}. An energy estimate comparing two self-consistent state-adjoint systems then shows that their state laws, and subsequently their optimal controls, must coincide.

\subsection{Well-posedness of the Reduced FBSDE and the Consistency Map}
\label{subsec3.1}

We first establish the well-posedness of the reduced FBSDE
\eqref{eq:reduced_FBSDE} for a prescribed measure flow.

For $\mu\in\mathcal M$, the kernel-based mean field can be written as $ m_t^K[\mu] = \mathcal K a_t^\mu$, where $a_t^\mu$ is defined as in \eqref{1st_moment_profile}, and $\mathcal K$ is the integral operator introduced in Subsection \ref{subsec2.1}. By Jensen's inequality,
$$
\begin{aligned}
\|a_t^\mu\|_{L^2_\lambda}^2
& = \int_\Theta \left| \int_{\mathbb R^k} x\,\mu_t^\theta(\mathrm dx) \right|^2 \lambda(\mathrm d\theta)
\\
&\le \int_\Theta \int_{\mathbb R^k} |x|^2\mu_t^\theta(\mathrm dx) \lambda(\mathrm d\theta).
\end{aligned}
$$
Consequently,
\begin{equation}
\label{eq:mean_field_L2_bound}
\|m_t^K[\mu]\|_{L^2_\lambda}^2 \le \kappa_K^2 \int_\Theta \int_{\mathbb R^k} |x|^2\mu_t^\theta(\mathrm dx) \lambda(\mathrm d\theta).
\end{equation}
Moreover, for $s,t\in[0,T]$,
$$
\begin{aligned}
\|a_t^\mu-a_s^\mu\|_{L^2_\lambda}^2
&\le \int_\Theta \mathcal W_2 \bigl( \mu_t^\theta,\mu_s^\theta \bigr)^2 \lambda(\mathrm d\theta).
\end{aligned}
$$
Since $\mu\in\mathcal M$, it follows that $ a^\mu \in C\bigl([0,T];L^2(\Theta;\mathbb R^k)\bigr). $

The boundedness of $\mathcal K$ therefore yields
\begin{equation}
\label{eq:mean_field_path_space}
m^K[\mu] \in C\bigl([0,T];L^2(\Theta;\mathbb R^k)\bigr).
\end{equation}
In particular,
$$
\int_\Theta \left( |m_T^{\mu,\theta}|^2 + \int_0^T |m_t^{\mu,\theta}|^2 \mathrm dt \right) \lambda(\mathrm d\theta) <\infty,
$$
where, as before,
$$
m_t^{\mu,\theta} := m_t^K[\mu](\theta).
$$

We now apply the monotonicity method for fully coupled FBSDEs.

\begin{lemma}
\label{lem001}
Suppose that Assumptions \ref{asmp001}, \ref{asmp003}, and \ref{asmp006} hold. Then, for every $\mu\in\mathcal M$ and for $\lambda$-a.e. $\theta\in\Theta$, the reduced FBSDE \eqref{eq:reduced_FBSDE} admits a unique square-integrable adapted solution
$$
\left( X^{\mu,\theta}, p^{\mu,\theta}, q^{\mu,\theta} \right).
$$
The corresponding type-indexed family belongs to
$$
\mathbb S_\Theta^2 \times \mathbb S_\Theta^2 \times \mathbb H_\Theta^2.
$$
Moreover, there exists a constant $C_{\mathrm{F}}>0$ depending only on $T$, $\sigma$, and the constants in Assumption \ref{asmp006}, such that
\begin{equation}
\label{eq:reduced_FBSDE_estimate}
\begin{aligned}
& \int_\Theta \mathbb E \left[ \sup_{t\in[0,T]} |X_t^{\mu,\theta}|^2 + \sup_{t\in[0,T]} |p_t^{\mu,\theta}|^2 + \int_0^T |q_t^{\mu,\theta}|^2 \mathrm dt \right] \lambda(\mathrm d\theta)
\\
&\le C_{\mathrm F} \Bigg[ 1 + \int_\Theta \mathbb E|\xi^\theta|^2 \lambda(\mathrm d\theta) + \|m_T^K[\mu]\|_{L^2_\lambda}^2 + \int_0^T \|m_t^K[\mu]\|_{L^2_\lambda}^2 \mathrm dt \Bigg].
\end{aligned}
\end{equation}
\end{lemma}

\begin{proof}
Fix $\mu\in\mathcal M$. By \eqref{eq:mean_field_path_space}, for $\lambda$-a.e. $\theta\in\Theta$,
$$
|m_T^{\mu,\theta}|^2 + \int_0^T |m_t^{\mu,\theta}|^2 \mathrm dt <\infty.
$$
For such a $\theta$, define
$$
B_t^{\mu,\theta}(x,p) := B(t,x,p,m_t^{\mu,\theta},\theta),
$$
$$
F_t^{\mu,\theta}(x,p) := F(t,x,p,m_t^{\mu,\theta},\theta),
$$
and
$$
\Gamma^{\mu,\theta}(x) := \Gamma(x,m_T^{\mu,\theta},\theta).
$$

By Assumption \ref{asmp006}(a), $B^{\mu,\theta}$ and $F^{\mu,\theta}$ are uniformly Lipschitz in $(x,p)$ and $\Gamma^{\mu,\theta}$ is uniformly Lipschitz in $x$. Moreover,
$$
\begin{aligned}
& |B_t^{\mu,\theta}(0,0)| + |F_t^{\mu,\theta}(0,0)|
\\
&\le C\bigl(1+|m_t^{\mu,\theta}|\bigr),
\end{aligned}
$$
and
$$
|\Gamma^{\mu,\theta}(0)| \le C\bigl(1+|m_T^{\mu,\theta}|\bigr).
$$
Hence the coefficients at the origin satisfy the required square-integrability conditions.

Assumption \ref{asmp006}(b) further gives
$$
\begin{aligned}
& \left\langle B_t^{\mu,\theta}(x,p) - B_t^{\mu,\theta}(x',p'), p-p' \right\rangle
\\
&\quad - \left\langle F_t^{\mu,\theta}(x,p) - F_t^{\mu,\theta}(x',p'), x-x' \right\rangle
\\
&\ge \beta|x-x'|^2,
\end{aligned}
$$
and
$$
\left\langle \Gamma^{\mu,\theta}(x) - \Gamma^{\mu,\theta}(x'), x-x' \right\rangle \le -\nu|x-x'|^2.
$$
Thus the frozen coefficients satisfy the Peng-Wu monotonicity conditions in the present constant-diffusion setting. The continuation theorem for monotone fully coupled FBSDEs therefore yields a unique square-integrable adapted solution of \eqref{eq:reduced_FBSDE} for $\lambda$-a.e. $\theta$.

The corresponding a priori estimate gives
$$
\begin{aligned}
& \mathbb E \left[ \sup_{t\in[0,T]} |X_t^{\mu,\theta}|^2 + \sup_{t\in[0,T]} |p_t^{\mu,\theta}|^2 + \int_0^T |q_t^{\mu,\theta}|^2 \mathrm dt
\right]
\\
&\le C_{\mathrm F} \left( 1+\mathbb E|\xi^\theta|^2 + |m_T^{\mu,\theta}|^2 + \int_0^T |m_t^{\mu,\theta}|^2 \mathrm dt \right),
\end{aligned}
$$
where $C_{\mathrm F}$ is independent of $\theta$ and $\mu$. Integrating with respect to $\lambda(\mathrm d\theta)$ proves \eqref{eq:reduced_FBSDE_estimate}.
\end{proof}

\begin{remark}
\label{rem:prescribed_path_well-posedness}
The proof of Lemma \ref{lem001} uses the prescribed measure flow only through its kernel mean field. Hence the same argument applies to an arbitrary prescribed path $m\in C([0,T];L^2(\Theta;\mathbb R^k))$ and yields a unique solution of \eqref{eq:frozen_m_FBSDE}. 
\end{remark}

The parameterized solution may moreover be taken jointly measurable in the type variable. Indeed, in the continuation construction for the frozen FBSDE, the initial system has a jointly measurable solution by Assumptions~\ref{asmp001} and \ref{asmp006}(a). At each continuation step, the Picard map involves only the jointly measurable coefficients $B$, $F$, and $\Gamma$, together with conditional expectations and stochastic integrals, and therefore preserves joint measurability in $(\theta,t,\omega)$. Hence every iterate admits a jointly measurable version. Since the contraction and a priori estimates used in the continuation argument are uniform in $\theta$, the iterates converge in $\mathbb S_\Theta^2\times\mathbb S_\Theta^2\times\mathbb H_\Theta^2$. In particular, all three components converge in the corresponding product-space $L^2$ norms on $\Theta\times[0,T]\times\Omega$. Passing to a subsequence therefore gives $\lambda\otimes\mathrm dt\otimes\mathbb P$-a.e. convergence. As pointwise limits of jointly measurable functions are jointly measurable, the limiting solution admits a jointly measurable version. By uniqueness, this version is the parameterized solution $(X^{\mu,\theta},p^{\mu,\theta},q^{\mu,\theta})$ obtained above.

We may therefore define, for every $\mu\in\mathcal M$,
\begin{equation}
\label{eq:Phi}
\Phi(\mu) := \left\{ \Phi(\mu)_t^\theta \right\}_{t\in[0,T],\,\theta\in\Theta}, 
\qquad
 \Phi(\mu)_t^\theta := \mathcal L(X_t^{\mu,\theta}).
\end{equation}
We refer to $\Phi$ as the \emph{consistency map}. At this stage, this terminology refers only to consistency between the prescribed measure flow and the state distribution generated by the reduced FBSDE; no optimality property is being asserted.

It remains to verify that $\Phi(\mu)$ belongs to $\mathcal M$. First, \eqref{eq:reduced_FBSDE_estimate} yields
$$
\sup_{t\in[0,T]} \int_\Theta \mathbb E |X_t^{\mu,\theta}|^2 \lambda(\mathrm d\theta) <\infty.
$$
For $0\le s\le t\le T$,
$$
\begin{aligned}
X_t^{\mu,\theta}-X_s^{\mu,\theta} ={}
& \int_s^t B(r,X_r^{\mu,\theta},p_r^{\mu,\theta}, m_r^{\mu,\theta},\theta) \mathrm dr
\\
& + \sigma(W_t-W_s).
\end{aligned}
$$
Consequently,
$$
\begin{aligned}
& \int_\Theta \mathbb E |X_t^{\mu,\theta}-X_s^{\mu,\theta}|^2 \lambda(\mathrm d\theta)
\\
&\le 2|t-s| \int_s^t \int_\Theta \mathbb E \left| B(r,X_r^{\mu,\theta},p_r^{\mu,\theta}, m_r^{\mu,\theta},\theta) \right|^2 \lambda(\mathrm d\theta) \mathrm dr
\\
&\quad + 2\|\sigma\|_{\mathrm F}^2|t-s|.
\end{aligned}
$$
By Assumption \ref{asmp006}(a), \eqref{eq:mean_field_path_space}, and \eqref{eq:reduced_FBSDE_estimate}, the integral on the right-hand side is finite. Hence there exists a finite constant $C_\mu$ such that
\begin{equation}
\label{eq:Phi_time_modulus}
\int_\Theta \mathbb E |X_t^{\mu,\theta}-X_s^{\mu,\theta}|^2 \lambda(\mathrm d\theta) \le C_\mu|t-s|.
\end{equation}
Using the joint realization $(X_t^{\mu,\theta},X_s^{\mu,\theta})$ as a coupling gives
$$
\begin{aligned}
& \int_\Theta \mathcal W_2 \left( \Phi(\mu)_t^\theta, \Phi(\mu)_s^\theta \right)^2 \lambda(\mathrm d\theta)
\\
&\le C_\mu|t-s|.
\end{aligned}
$$
Therefore $\Phi(\mu)\in\mathcal M$, and the consistency map $\Phi:\mathcal M\longrightarrow\mathcal M$ is well-defined.

The quantitative smallness condition on $\kappa_K$ in Assumption \ref{asmp003} is not used in the preceding well-posedness argument. Only the boundedness of $\mathcal K$ on $L^2(\Theta;\mathbb R^k)$ is needed here. 

Next, we verify optimality for the frozen control problem.

\begin{proposition}
\label{prop:frozen_optimizer}
Suppose that Assumptions \ref{asmp001}--\ref{asmp006} hold. Let
$$
m\in C([0,T];L^2(\Theta;\mathbb R^k))
$$
be a prescribed mean-field path, and let $(X^{m,\theta},p^{m,\theta},q^{m,\theta})$ denote the unique solution of the corresponding reduced FBSDE. Then, for $\lambda$-a.e. $\theta$, the feedback control
\begin{equation}
\label{eq:frozen_optimal_control}
u_t^{m,\theta} = \hat u\bigl( t,X_t^{m,\theta},p_t^{m,\theta}, m_t(\theta),\theta \bigr)
\end{equation}
is the unique optimal admissible control for the stochastic control problem with the prescribed mean-field path $m$.
\end{proposition}

\begin{proof}
Fix $\theta\in\Theta$ outside a $\lambda$-null set for which the preceding FBSDE estimates hold, and write
$$
(X,p,q,u) = (X^{m,\theta},p^{m,\theta},q^{m,\theta},u^{m,\theta})
$$
for simplicity. Let $\widetilde u$ be any admissible control, and let $\widetilde X$ be the corresponding state process under the same prescribed mean-field path $m$. Set
$$
\Delta X_t:=\widetilde X_t-X_t.
$$

We first verify that the feedback control $u$ is admissible. By Assumption~\ref{asmp004}(c), $\hat u$ is jointly measurable. Since $(X,p)$ are adapted continuous processes and $t\mapsto m_t(\theta)$ is deterministic and measurable, it follows that
$$
u_t = \hat u\bigl(t,X_t,p_t,m_t(\theta),\theta\bigr)
$$
is progressively measurable.

Moreover, by Assumption~\ref{asmp005}(a),
$$
|u_t|^2 \le C\bigl(1+|X_t|^2+|p_t|^2+|m_t(\theta)|^2\bigr).
$$
Hence, using the a priori estimate for the frozen FBSDE,
$$
\mathbb E\int_0^T |u_t|^2\,\mathrm dt \le C\left( 1+ \mathbb E\sup_{t\in[0,T]}|X_t|^2 + \mathbb E\sup_{t\in[0,T]}|p_t|^2 + \int_0^T|m_t(\theta)|^2\,\mathrm dt \right) <\infty.
$$
Indeed,
$$
\int_\Theta\int_0^T |m_t(\theta)|^2\,\mathrm dt\,\lambda(\mathrm d\theta) \le T\sup_{t\in[0,T]}\|m_t\|_{L^2_\lambda}^2 <\infty,
$$
so that $m_\cdot(\theta)\in L^2(0,T;\mathbb R^k)$ for $\lambda$-a.e. $\theta$. Thus $u$ is admissible.

Since both controls are evaluated against the same exogenous path $m_t(\theta)$, the difference of their objective functionals is
$$
\begin{aligned}
J^\theta(\widetilde u;m)-J^\theta(u;m) = {}
& \mathbb E\int_0^T \Big[ f\bigl(t,\widetilde X_t,\widetilde u_t, m_t(\theta),\theta\bigr) - f\bigl(t,X_t,u_t,m_t(\theta),\theta\bigr) \Big] \,\mathrm dt
\\
&+ \mathbb E\Big[ g\bigl(\widetilde X_T,m_T(\theta),\theta\bigr) - g\bigl(X_T,m_T(\theta),\theta\bigr) \Big].
\end{aligned}
$$

By the concavity of $x\mapsto g(x,m_T(\theta),\theta)$,
$$
\begin{aligned}
& g\bigl(\widetilde X_T,m_T(\theta),\theta\bigr) - g\bigl(X_T,m_T(\theta),\theta\bigr)
\\
&\qquad\le \partial_x g\bigl(X_T,m_T(\theta),\theta\bigr) \cdot\Delta X_T = p_T\cdot\Delta X_T.
\end{aligned}
$$
Consequently,
\begin{equation}
\label{eq:frozen_verification_1}
\begin{aligned}
J^\theta(\widetilde u;m)-J^\theta(u;m) \le{}
& \mathbb E\int_0^T \Big[ f\bigl(t,\widetilde X_t,\widetilde u_t, m_t(\theta),\theta\bigr) - f\bigl(t,X_t,u_t,m_t(\theta),\theta\bigr) \Big] \,\mathrm dt
\\
& + \mathbb E[p_T\cdot\Delta X_T].
\end{aligned}
\end{equation}

Applying It\^o's formula to $p_t\cdot\Delta X_t$, using $\Delta X_0=0$, and taking expectations, we obtain
$$
\begin{aligned}
\mathbb E[p_T\cdot\Delta X_T] ={}
& \mathbb E\int_0^T p_t\cdot \Big[ b\bigl(t,\widetilde X_t,\widetilde u_t, m_t(\theta),\theta\bigr) - b\bigl(t,X_t,u_t,m_t(\theta),\theta\bigr) \Big] \,\mathrm dt
\\
& - \mathbb E\int_0^T \partial_x H\bigl( t,X_t,u_t,p_t; \theta,m_t(\theta) \bigr) \cdot\Delta X_t \,\mathrm dt.
\end{aligned}
$$
Here the stochastic-integral terms have zero expectation, and the quadratic covariation term vanishes because the diffusion coefficient is independent of the state and the control.

Substituting this identity into \eqref{eq:frozen_verification_1} and recalling the definition of the Hamiltonian gives
\begin{equation}
\label{eq:frozen_verification_2}
\begin{aligned}
J^\theta(\widetilde u;m)-J^\theta(u;m) \le \mathbb E\int_0^T \Big[ 
& H\bigl( t,\widetilde X_t,\widetilde u_t,p_t; \theta,m_t(\theta) \bigr) - H\bigl( t,X_t,u_t,p_t; \theta,m_t(\theta) \bigr)
\\
& - \partial_x H\bigl( t,X_t,u_t,p_t; \theta,m_t(\theta) \bigr) \cdot (\widetilde X_t-X_t) \Big] \,\mathrm dt.
\end{aligned}
\end{equation}

Since
$$
H\bigl( t,X_t,u_t,p_t; \theta,m_t(\theta) \bigr) = \widehat H\bigl( t,X_t,p_t;\theta,m_t(\theta) \bigr),
$$
the strong concavity of $\widehat H$ in the state variable,
established in Remark \ref{rem:reduced_concavity}, yields
\begin{equation}
\label{eq:frozen_verification_3}
\begin{aligned}
& H\bigl( t,\widetilde X_t,\widetilde u_t,p_t; \theta,m_t(\theta) \bigr) - H\bigl( t,X_t,u_t,p_t; \theta,m_t(\theta) \bigr)
\\
&\qquad - \left\langle \partial_x H\bigl( t,X_t,u_t,p_t; \theta,m_t(\theta) \bigr), \widetilde X_t-X_t \right\rangle
\\
& = H\bigl( t,\widetilde X_t,\widetilde u_t,p_t; \theta,m_t(\theta) \bigr) - \widehat H\bigl( t,\widetilde X_t,p_t;\theta,m_t(\theta) \bigr) 
\\
&\qquad + \widehat H\bigl( t,\widetilde X_t,p_t;\theta,m_t(\theta) \bigr) - \widehat H\bigl( t,X_t,p_t;\theta,m_t(\theta) \bigr)
\\
&\qquad - \left\langle \partial_x H\bigl( t,X_t,u_t,p_t; \theta,m_t(\theta) \bigr), \widetilde X_t-X_t \right\rangle
\\
&\le -\frac{\beta}{2} |\widetilde X_t-X_t|^2 + H\bigl( t,\widetilde X_t,\widetilde u_t,p_t; \theta,m_t(\theta) \bigr) - \widehat H\bigl( t,\widetilde X_t,p_t;\theta,m_t(\theta) \bigr), 
\end{aligned}
\end{equation}
where Assumption \ref{asmp005}(b) gives
$$
\partial_x H\bigl( t,X_t,u_t,p_t; \theta,m_t(\theta) \bigr) = \partial_x\widehat H\bigl( t,X_t,p_t;\theta,m_t(\theta) \bigr).
$$
Since
$$
H\bigl( t,\widetilde X_t,\widetilde u_t,p_t; \theta,m_t(\theta) \bigr) \le \widehat H\bigl( t,\widetilde X_t,p_t;\theta,m_t(\theta) \bigr),
$$
combining this inequality with \eqref{eq:frozen_verification_2} and \eqref{eq:frozen_verification_3} gives 
$$
\begin{aligned}
& J^\theta(\widetilde u;m)-J^\theta(u;m) 
\\
&\le{}
-\frac{\beta}{2} \mathbb E\int_0^T |\widetilde X_t-X_t|^2\,\mathrm dt 
\\
&\quad - \mathbb E\int_0^T \Big[\widehat H\bigl( t,\widetilde X_t,p_t;\theta,m_t(\theta) \bigr) - H\bigl( t,\widetilde X_t,\widetilde u_t,p_t; \theta,m_t(\theta) \bigr)\Big]\,\mathrm dt 
\\
&\le 0.
\end{aligned}
$$
Since $\widetilde u$ was arbitrary, $u^{m,\theta}$ is optimal for the stochastic control problem with prescribed mean-field path $m$.

If $\widetilde u$ is also optimal, then
$$
\begin{aligned}
0 
& = J^\theta(\widetilde u;m)-J^\theta(u;m) 
\\
&\le{}
-\frac{\beta}{2} \mathbb E\int_0^T |\widetilde X_t-X_t|^2\,\mathrm dt 
\\
&\quad - \mathbb E\int_0^T \Big[\widehat H\bigl( t,\widetilde X_t,p_t;\theta,m_t(\theta) \bigr) - H\bigl( t,\widetilde X_t,\widetilde u_t,p_t; \theta,m_t(\theta) \bigr)\Big]\,\mathrm dt 
\\
& \le 0.
\end{aligned}
$$
Since both integrands above are non-negative, it follows that 
$$
\widetilde X_t = X_t, 
\qquad 
\widehat H\bigl( t,\widetilde X_t,p_t;\theta,m_t(\theta) \bigr) = H\bigl( t,\widetilde X_t,\widetilde u_t,p_t; \theta,m_t(\theta) \bigr).
$$
By Assumption \ref{asmp004}(a) and (b), the Hamiltonian maximizer $\hat u$ is unique. Therefore $\widetilde u=u$, $\mathrm dt\otimes\mathrm d\mathbb P$-a.e.
\end{proof}

\subsection{Stability of the Consistency Map}
\label{subsec3.2}

We next establish a quantitative stability estimate for the consistency map. The main point is that the dependence of the reduced FBSDE on the prescribed measure flow occurs only through the kernel-based mean field.

Let $\mu,\nu\in\mathcal M$, and define
$$
m_t^\mu := m_t^K[\mu],
\qquad
m_t^\nu := m_t^K[\nu].
$$
For $\lambda$-a.e. $\theta$, the first-moment profiles of $\mu$, $\nu$ satisfy
$$
|a_t^\mu(\theta)-a_t^\nu(\theta)|
\le
\mathcal W_2
\left(
\mu_t^\theta,\nu_t^\theta
\right).
$$
It follows that
$$
\begin{aligned}
\|a_t^\mu-a_t^\nu\|_{L^2_\lambda}^2
&\le
\int_\Theta
\mathcal W_2
\left(
\mu_t^\theta,\nu_t^\theta
\right)^2
\lambda(\mathrm d\theta).
\end{aligned}
$$
Since
$$
m_t^\mu-m_t^\nu
=
\mathcal K(a_t^\mu-a_t^\nu),
$$
Assumption \ref{asmp003} yields
\begin{equation}
\label{eq:mean_field_stability}
\begin{aligned}
\|m_t^\mu-m_t^\nu\|_{L^2_\lambda}
&\le \kappa_K \|a_t^\mu-a_t^\nu\|_{L^2_\lambda}
\\
&\le \kappa_K \left( \int_\Theta \mathcal W_2 \left( \mu_t^\theta,\nu_t^\theta \right)^2 \lambda(\mathrm d\theta) \right)^{1/2}
\\
&\le \kappa_K d(\mu,\nu).
\end{aligned}
\end{equation}

We now derive the corresponding stability estimate for the reduced FBSDE.

\begin{lemma}
\label{lem002}
Suppose that Assumption \ref{asmp006} holds. Let
$$
(X^{\mu,\theta},p^{\mu,\theta},q^{\mu,\theta})
$$
and
$$
(X^{\nu,\theta},p^{\nu,\theta},q^{\nu,\theta})
$$
be the solutions of \eqref{eq:reduced_FBSDE} corresponding to $\mu$ and $\nu$, respectively. Then there exists a constant
$$
C_{\mathrm{stab}}>0,
$$
depending only on $T$ and the constants in Assumption \ref{asmp006}, such that
\begin{equation}
\label{eq:FBSDE_stability}
\begin{aligned}
&
\int_\Theta
\mathbb E
\sup_{t\in[0,T]}
|X_t^{\mu,\theta}-X_t^{\nu,\theta}|^2
\lambda(\mathrm d\theta)
\\
&\quad+
\int_\Theta
\mathbb E
\sup_{t\in[0,T]}
|p_t^{\mu,\theta}-p_t^{\nu,\theta}|^2
\lambda(\mathrm d\theta)
\\
&\quad+
\int_\Theta
\mathbb E
\int_0^T
|q_t^{\mu,\theta}-q_t^{\nu,\theta}|^2
\mathrm dt\,
\lambda(\mathrm d\theta)
\\
&\le
C_{\mathrm{stab}}
\left[
\|m_T^\mu-m_T^\nu\|_{L^2_\lambda}^2
+
\int_0^T
\|m_t^\mu-m_t^\nu\|_{L^2_\lambda}^2
\mathrm dt
\right].
\end{aligned}
\end{equation}
\end{lemma}

\begin{proof}
For $\lambda$-a.e. $\theta$, set
$$
\Delta X_t^\theta
:=
X_t^{\mu,\theta}
-
X_t^{\nu,\theta},
\qquad
\Delta p_t^\theta
:=
p_t^{\mu,\theta}
-
p_t^{\nu,\theta},
$$
$$
\Delta q_t^\theta
:=
q_t^{\mu,\theta}
-
q_t^{\nu,\theta},
\qquad
\Delta m_t^\theta
:=
m_t^{\mu,\theta}
-
m_t^{\nu,\theta}.
$$
Since the two forward equations have the same initial condition and the same constant diffusion coefficient,
$$
\Delta X_0^\theta=0,
$$
and
$$
\mathrm d\Delta X_t^\theta
=
\Delta B_t^\theta\,\mathrm dt,
$$
where
$$
\begin{aligned}
\Delta B_t^\theta
:={}&
B(t,X_t^{\mu,\theta},p_t^{\mu,\theta},
m_t^{\mu,\theta},\theta)
\\
&-
B(t,X_t^{\nu,\theta},p_t^{\nu,\theta},
m_t^{\nu,\theta},\theta).
\end{aligned}
$$
Similarly, set
$$
\begin{aligned}
\Delta F_t^\theta
:={}&
F(t,X_t^{\mu,\theta},p_t^{\mu,\theta},
m_t^{\mu,\theta},\theta)
\\
&-
F(t,X_t^{\nu,\theta},p_t^{\nu,\theta},
m_t^{\nu,\theta},\theta).
\end{aligned}
$$

By Assumption \ref{asmp006}(a), the standard stability estimate for Lipschitz backward stochastic differential equations gives
\begin{equation}
\label{eq:BSDE_stability_aux}
\begin{aligned}
&
\mathbb E
\sup_{t\in[0,T]}
|\Delta p_t^\theta|^2
+
\mathbb E
\int_0^T
|\Delta q_t^\theta|^2
\mathrm dt
\\
&\le
C
\Bigg[
\mathbb E|\Delta X_T^\theta|^2
+
\mathbb E\int_0^T
|\Delta X_t^\theta|^2
\mathrm dt
+
|\Delta m_T^\theta|^2
+
\int_0^T
|\Delta m_t^\theta|^2
\mathrm dt
\Bigg].
\end{aligned}
\end{equation}

We next exploit the monotonicity structure. Applying It\^o's formula to
$$
\langle\Delta X_t^\theta,\Delta p_t^\theta\rangle
$$
and taking expectations yields
\begin{equation}
\label{eq:energy_identity}
\mathbb E
\langle
\Delta X_T^\theta,
\Delta p_T^\theta
\rangle
=
\mathbb E
\int_0^T
\left[
\langle\Delta B_t^\theta,\Delta p_t^\theta\rangle
-
\langle\Delta F_t^\theta,\Delta X_t^\theta\rangle
\right]
\mathrm dt.
\end{equation}
There is no quadratic-covariation term because $\Delta X^\theta$ has no martingale part.

To separate the variation in $(x,p)$ from the variation in the mean field, insert and subtract the coefficients evaluated at the common mean field $m_t^{\mu,\theta}$. Assumption \ref{asmp006} then gives
\begin{equation}
\label{eq:energy_drift_bound}
\begin{aligned}
&
\langle\Delta B_t^\theta,\Delta p_t^\theta\rangle
-
\langle\Delta F_t^\theta,\Delta X_t^\theta\rangle
\\
&\ge
\beta|\Delta X_t^\theta|^2
-
C|\Delta m_t^\theta|
\left(
|\Delta X_t^\theta|
+
|\Delta p_t^\theta|
\right).
\end{aligned}
\end{equation}
Likewise,
$$
\Delta p_T^\theta
=
\Gamma(X_T^{\mu,\theta},m_T^{\mu,\theta},\theta)
-
\Gamma(X_T^{\nu,\theta},m_T^{\nu,\theta},\theta),
$$
and hence
\begin{equation}
\label{eq:energy_terminal_bound}
\left\langle
\Delta X_T^\theta,
\Delta p_T^\theta
\right\rangle
\le
-\nu|\Delta X_T^\theta|^2
+
C|\Delta X_T^\theta||\Delta m_T^\theta|.
\end{equation}

Combining \eqref{eq:energy_identity}, \eqref{eq:energy_drift_bound}, and \eqref{eq:energy_terminal_bound}, and then applying Young's inequality, gives, for every $\varepsilon>0$,
$$
\begin{aligned}
&
\nu\mathbb E|\Delta X_T^\theta|^2
+
\beta
\mathbb E\int_0^T
|\Delta X_t^\theta|^2
\mathrm dt
\\
&\le
\varepsilon
\left[
\mathbb E|\Delta X_T^\theta|^2
+
\mathbb E\int_0^T
|\Delta X_t^\theta|^2
\mathrm dt
+
\mathbb E\int_0^T
|\Delta p_t^\theta|^2
\mathrm dt
\right]
\\
&\quad+
C_\varepsilon
\left[
|\Delta m_T^\theta|^2
+
\int_0^T
|\Delta m_t^\theta|^2
\mathrm dt
\right].
\end{aligned}
$$
Using \eqref{eq:BSDE_stability_aux} and choosing $\varepsilon>0$ sufficiently small, the terms involving $\Delta X^\theta$ can be absorbed into the left-hand side. We obtain
\begin{equation}
\label{eq:X_integral_stability}
\begin{aligned}
&
\mathbb E|\Delta X_T^\theta|^2
+
\mathbb E\int_0^T
|\Delta X_t^\theta|^2
\mathrm dt
\\
&\le
C
\left[
|\Delta m_T^\theta|^2
+
\int_0^T
|\Delta m_t^\theta|^2
\mathrm dt
\right].
\end{aligned}
\end{equation}
Substituting this into \eqref{eq:BSDE_stability_aux} yields
\begin{equation}
\label{eq:pq_stability}
\begin{aligned}
&
\mathbb E
\sup_{t\in[0,T]}
|\Delta p_t^\theta|^2
+
\mathbb E
\int_0^T
|\Delta q_t^\theta|^2
\mathrm dt
\\
&\le
C
\left[
|\Delta m_T^\theta|^2
+
\int_0^T
|\Delta m_t^\theta|^2
\mathrm dt
\right].
\end{aligned}
\end{equation}

Finally,
$$
\Delta X_t^\theta
=
\int_0^t
\Delta B_s^\theta
\mathrm ds,
$$
so that
$$
\begin{aligned}
\mathbb E
\sup_{t\in[0,T]}
|\Delta X_t^\theta|^2
&\le
T
\mathbb E
\int_0^T
|\Delta B_t^\theta|^2
\mathrm dt
\\
&\le
C
\mathbb E
\int_0^T
\left(
|\Delta X_t^\theta|^2
+
|\Delta p_t^\theta|^2
+
|\Delta m_t^\theta|^2
\right)
\mathrm dt.
\end{aligned}
$$
Together with \eqref{eq:X_integral_stability} and \eqref{eq:pq_stability}, this gives
$$
\begin{aligned}
&
\mathbb E
\sup_{t\in[0,T]}
|\Delta X_t^\theta|^2
+
\mathbb E
\sup_{t\in[0,T]}
|\Delta p_t^\theta|^2
+
\mathbb E
\int_0^T
|\Delta q_t^\theta|^2
\mathrm dt
\\
&\le
C
\left[
|\Delta m_T^\theta|^2
+
\int_0^T
|\Delta m_t^\theta|^2
\mathrm dt
\right].
\end{aligned}
$$
The constant $C$ is independent of $\theta$. Integrating with respect to $\lambda(\mathrm d\theta)$ proves \eqref{eq:FBSDE_stability}.
\end{proof}

The preceding estimate immediately gives the continuity of the consistency map.

\begin{proposition}
\label{prop:Phi_stability}
Suppose that Assumptions \ref{asmp001}, \ref{asmp003}, and \ref{asmp006} hold. Then there exists a constant $C_\Phi>0$ such that
\begin{equation}
\label{eq:Phi_Lipschitz}
d(\Phi(\mu),\Phi(\nu))
\le
C_\Phi\kappa_K\,d(\mu,\nu),
\qquad
\mu,\nu\in\mathcal M,
\end{equation}
where $C_\Phi$ depends only on $T$ and the constants in Assumption \ref{asmp006}. In particular,
$$
\Phi:\mathcal M\longrightarrow\mathcal M
$$
is continuous.
\end{proposition}

\begin{proof}
For each $(t,\theta)$,
$$
(X_t^{\mu,\theta},X_t^{\nu,\theta})
$$
is a coupling of $\Phi(\mu)_t^\theta$ and $\Phi(\nu)_t^\theta$. Hence
$$
\mathcal W_2
\left(
\Phi(\mu)_t^\theta,
\Phi(\nu)_t^\theta
\right)^2
\le
\mathbb E
|X_t^{\mu,\theta}-X_t^{\nu,\theta}|^2.
$$
Lemma \ref{lem002} therefore implies
$$
\begin{aligned}
d(\Phi(\mu),\Phi(\nu))^2
&\le
\int_\Theta
\mathbb E
\sup_{t\in[0,T]}
|X_t^{\mu,\theta}-X_t^{\nu,\theta}|^2
\lambda(\mathrm d\theta)
\\
&\le
C_{\mathrm{stab}}
\left[
\|m_T^\mu-m_T^\nu\|_{L^2_\lambda}^2
+
\int_0^T
\|m_t^\mu-m_t^\nu\|_{L^2_\lambda}^2
\mathrm dt
\right].
\end{aligned}
$$
Using \eqref{eq:mean_field_stability},
$$
\begin{aligned}
d(\Phi(\mu),\Phi(\nu))^2
&\le
C_{\mathrm{stab}}
(1+T)
\kappa_K^2
d(\mu,\nu)^2.
\end{aligned}
$$
Thus \eqref{eq:Phi_Lipschitz} holds with, for instance,
$$
C_\Phi
=
\sqrt{C_{\mathrm{stab}}(1+T)}.
$$
\end{proof}

The estimate above is used only to establish continuity of $\Phi$. In particular, the argument does not require
$$
C_\Phi\kappa_K<1
$$
and does not rely on a contraction principle. The quantitative smallness condition in Assumption \ref{asmp003} will instead be used in the construction of an invariant closed convex set for the Schauder fixed-point argument.

\subsection{Existence of a Mean Field Nash Equilibrium}
\label{subsec3.3}

We now establish the existence of a mean field Nash equilibrium. Compactness of $\Phi$ on the space of type-indexed measure flows does not follow directly from the present assumptions, since no regularity with respect to the type variable $\theta$ is imposed. The kernel structure provides the required compactness of the aggregate mean field. We therefore perform the fixed-point argument on $C([0,T];L^2(\Theta;\mathbb R^k))$ and then recover a fixed point of $\Phi$.

For a prescribed path
$$
m\in C([0,T];L^2(\Theta;\mathbb R^k)),
$$
consider the FBSDE
\begin{equation}
\label{eq:frozen_m_FBSDE}
\left\{
\begin{aligned}
\mathrm dX_t^{m,\theta}
& = B\bigl( t,X_t^{m,\theta},p_t^{m,\theta}, m_t(\theta),\theta \bigr)\mathrm dt + \sigma\,\mathrm dW_t,
\\
\mathrm dp_t^{m,\theta}
& = -F\bigl( t,X_t^{m,\theta},p_t^{m,\theta}, m_t(\theta),\theta \bigr)\mathrm dt + q_t^{m,\theta}\,\mathrm dW_t,
\\
X_0^{m,\theta}
& = \xi^\theta,
\\
p_T^{m,\theta}
& = \Gamma\bigl( X_T^{m,\theta},m_T(\theta),\theta \bigr).
\end{aligned}
\right.
\end{equation}
In particular, there exist constants $C_1,C_2>0$, independent of $m$, such that
\begin{equation}
\label{eq:frozen_m_apriori}
\begin{aligned}
&\int_\Theta \mathbb E\left[ \sup_{t\in[0,T]}|X_t^{m,\theta}|^2 + \sup_{t\in[0,T]}|p_t^{m,\theta}|^2 + \int_0^T|q_t^{m,\theta}|^2\,\mathrm dt \right] \lambda(\mathrm d\theta)
\\
&\qquad\le C_1 + C_2\left( \|m_T\|_{L^2_\lambda}^2 + \int_0^T\|m_t\|_{L^2_\lambda}^2\,\mathrm dt \right).
\end{aligned}
\end{equation}

Define
$$
\mathcal T(m)_t := \mathcal K \left( \theta\mapsto \mathbb E[X_t^{m,\theta}] \right),
\qquad
t\in[0,T].
$$
The next lemma collects the properties required for the fixed-point argument.

\begin{lemma}
\label{lem003}
Suppose that Assumptions \ref{asmp001}, \ref{asmp003}, and
\ref{asmp006} hold. If
\begin{equation}
\label{eq:kernel_invariance_smallness}
C_2(1+T)\kappa_K^2<1,
\end{equation}
where $C_2$ is the constant in \eqref{eq:frozen_m_apriori}, then there exists $R>0$ such that $\mathcal T$ maps
$$
B_R := \left\{ m\in C([0,T];L^2(\Theta;\mathbb R^k)): \sup_{t\in[0,T]}\|m_t\|_{L^2_\lambda}^2\le R \right\}
$$
into itself. Moreover, $\mathcal T:B_R\longrightarrow B_R$ is continuous and has relatively compact image.
\end{lemma}

\begin{proof}
We first prove the invariance of $B_R$. By Jensen's inequality and the boundedness of $\mathcal K$,
$$
\begin{aligned}
\|\mathcal T(m)_t\|_{L^2_\lambda}^2
&\le \kappa_K^2 \int_\Theta \left| \mathbb E[X_t^{m,\theta}] \right|^2 \lambda(\mathrm d\theta)
\\
&\le \kappa_K^2 \int_\Theta \mathbb E \left[ \sup_{s\in[0,T]}|X_s^{m,\theta}|^2 \right] \lambda(\mathrm d\theta).
\end{aligned}
$$
If $m\in B_R$, then \eqref{eq:frozen_m_apriori} gives
$$
\sup_{t\in[0,T]} \|\mathcal T(m)_t\|_{L^2_\lambda}^2 \le \kappa_K^2 C_1 + C_2(1+T)\kappa_K^2 R.
$$
In view of \eqref{eq:kernel_invariance_smallness}, we may choose
$$
R \ge \frac{\kappa_K^2C_1} {1-C_2(1+T)\kappa_K^2}.
$$
Then $\mathcal T(B_R)\subset B_R$.

We next prove continuity. Let $m,\widetilde m\in B_R$, and let $(X^m,p^m,q^m)$ and $(X^{\widetilde m},p^{\widetilde m},q^{\widetilde m})$ be the corresponding solutions of \eqref{eq:frozen_m_FBSDE}. The stability argument in Lemma \ref{lem002}, applied directly to the two prescribed mean-field paths, yields
$$
\begin{aligned}
&\int_\Theta \mathbb E\left[ \sup_{t\in[0,T]} |X_t^{m,\theta}-X_t^{\widetilde m,\theta}|^2 \right] \lambda(\mathrm d\theta)
\\
&\qquad\le C_{\mathrm{stab}} \left( \|m_T-\widetilde m_T\|_{L^2_\lambda}^2 + \int_0^T \|m_t-\widetilde m_t\|_{L^2_\lambda}^2 \,\mathrm dt \right)
\\
&\qquad\le C_{\mathrm{stab}}(1+T) \sup_{t\in[0,T]} \|m_t-\widetilde m_t\|_{L^2_\lambda}^2.
\end{aligned}
$$
Consequently,
$$
\begin{aligned}
\sup_{t\in[0,T]} \|\mathcal T(m)_t-\mathcal T(\widetilde m)_t\|_{L^2_\lambda}^2
&\le \kappa_K^2 \int_\Theta \mathbb E\left[ \sup_{t\in[0,T]} |X_t^{m,\theta}-X_t^{\widetilde m,\theta}|^2 \right] \lambda(\mathrm d\theta)
\\
&\le C_{\mathrm{stab}}(1+T)\kappa_K^2 \sup_{t\in[0,T]} \|m_t-\widetilde m_t\|_{L^2_\lambda}^2.
\end{aligned}
$$
Thus $\mathcal T$ is continuous on $B_R$.

It remains to prove relative compactness of $\mathcal T(B_R)$. For every fixed $t\in[0,T]$, \eqref{eq:frozen_m_apriori} implies
$$
\sup_{m\in B_R} \int_\Theta \left| \mathbb E[X_t^{m,\theta}] \right|^2 \lambda(\mathrm d\theta) <\infty.
$$
Since $K\in L^2(\Theta\times\Theta)$, the operator
$$
\mathcal K: L^2(\Theta;\mathbb R^k) \longrightarrow L^2(\Theta;\mathbb R^k)
$$
is Hilbert-Schmidt and hence compact. Therefore, for every fixed $t\in[0,T]$, the set $\left\{ \mathcal T(m)_t: m\in B_R \right\}$ is relatively compact in $L^2(\Theta;\mathbb R^k)$.

We finally establish a uniform modulus of continuity in time. Taking expectations in the forward equation in \eqref{eq:frozen_m_FBSDE}, for $0\le s\le t\le T$ we obtain
$$
\mathbb E[X_t^{m,\theta}] - \mathbb E[X_s^{m,\theta}] = \int_s^t \mathbb E\left[ B\bigl( r,X_r^{m,\theta},p_r^{m,\theta}, m_r(\theta),\theta \bigr) \right] \mathrm dr.
$$
Hence, by Jensen's inequality and the Cauchy-Schwarz inequality,
$$
\begin{aligned}
&\int_\Theta \left| \mathbb E[X_t^{m,\theta}] - \mathbb E[X_s^{m,\theta}] \right|^2 \lambda(\mathrm d\theta)
\\
&\qquad\le |t-s| \int_s^t \int_\Theta \mathbb E \left| B\bigl( r,X_r^{m,\theta},p_r^{m,\theta}, m_r(\theta),\theta \bigr) \right|^2 \lambda(\mathrm d\theta) \,\mathrm dr.
\end{aligned}
$$
By the growth condition on $B$, the definition of $B_R$, and \eqref{eq:frozen_m_apriori}, the last integral is bounded uniformly over $m\in B_R$. Thus there exists a constant $C_R>0$ such that
$$
\int_\Theta \left| \mathbb E[X_t^{m,\theta}] - \mathbb E[X_s^{m,\theta}] \right|^2 \lambda(\mathrm d\theta) \le C_R|t-s|,
$$
for all $m\in B_R$. It follows that
$$
\|\mathcal T(m)_t-\mathcal T(m)_s\|_{L^2_\lambda}^2 \le \kappa_K^2 C_R|t-s|.
$$
Hence $\mathcal T(B_R)$ is uniformly equicontinuous as a family of $L^2(\Theta;\mathbb R^k)$-valued continuous functions.

We have shown that $\mathcal T(B_R)$ is pointwise relatively compact and uniformly equicontinuous. The Banach-valued Arzel\`a-Ascoli theorem therefore implies that $\mathcal T(B_R)$ is relatively compact in $C([0,T];L^2(\Theta;\mathbb R^k))$.
\end{proof}

\begin{remark}
\label{rem:smallness_role}
The smallness condition \eqref{eq:kernel_invariance_smallness} is used only to construct an invariant ball for the fixed-point argument. Neither the well-posedness of the frozen FBSDE nor the stability estimate in Subsections \ref{subsec3.1}--\ref{subsec3.2} requires such a restriction. In particular, this condition is not a contraction condition for the consistency map.
\end{remark}

We can now establish the existence of a mean field Nash equilibrium.

\begin{theorem}
\label{thrm001}
Suppose that Assumptions \ref{asmp001}--\ref{asmp006} hold, and that the smallness condition in Assumption \ref{asmp003} is chosen so that
$$
C_2(1+T)\kappa_K^2<1,
$$
where $C_2$ is the constant in \eqref{eq:frozen_m_apriori}. Then there exists at least one mean
field Nash equilibrium.
\end{theorem}

\begin{proof}
Let $R>0$ be chosen as in the proof of Lemma \ref{lem003}. The set $B_R$ is a nonempty closed convex subset of the Banach space $C([0,T];L^2(\Theta;\mathbb R^k))$. By Lemma \ref{lem003}, $\mathcal T:B_R\longrightarrow B_R$ is continuous and has relatively compact image. Schauder's fixed-point theorem therefore yields $m^\ast\in B_R$ such that 
\begin{equation}
\label{eq:mean_field_fixed_point}
m_t^\ast = \mathcal K \left( \theta\mapsto \mathbb E[X_t^{m^\ast,\theta}] \right),
\qquad
t\in[0,T].
\end{equation}

Let
$$
(X^{\ast,\theta},p^{\ast,\theta},q^{\ast,\theta})
:=
(X^{m^\ast,\theta},p^{m^\ast,\theta},q^{m^\ast,\theta})
$$
and define
$$
\mu_t^{\ast,\theta} := \mathcal L(X_t^{\ast,\theta}).
$$
The estimates in Lemma \ref{lem001} imply that $\mu^\ast\in\mathcal M$. Moreover, \eqref{eq:mean_field_fixed_point} gives
\begin{equation}
\label{eq:mstar_consistency}
m_t^\ast = \mathcal K \left( \theta\mapsto \int_{\mathbb R^k}x\, \mu_t^{\ast,\theta}(\mathrm dx) \right) = m_t^K[\mu^\ast].
\end{equation}
Consequently, $(X^\ast,p^\ast,q^\ast)$ solves the reduced FBSDE \eqref{eq:reduced_FBSDE} associated with the measure flow $\mu^\ast$. By the uniqueness in Lemma \ref{lem001}, its forward component coincides with the one used in the definition of $\Phi(\mu^\ast)$. Hence
\begin{equation}
\label{eq:Phi_fixed_point}
\Phi(\mu^\ast)_t^\theta = \mathcal L(X_t^{\ast,\theta}) = \mu_t^{\ast,\theta},
\end{equation}
for every $t\in[0,T]$ and $\lambda$-a.e. $\theta$.

It remains to verify that the feedback associated with this consistency fixed point is optimal. Define
$$
u_t^{\ast,\theta} = \hat u\bigl( t,X_t^{\ast,\theta},p_t^{\ast,\theta}, m_t^\ast(\theta),\theta \bigr).
$$
By \eqref{eq:mstar_consistency}, the triple $(X^{\ast,\theta},p^{\ast,\theta},q^{\ast,\theta})$ is the reduced FBSDE solution corresponding to the frozen environment $m^K[\mu^\ast]$. Proposition \ref{prop:frozen_optimizer} therefore shows that $u^{\ast,\theta}$ is the optimal admissible control for the frozen stochastic control problem for $\lambda$-a.e. $\theta$. Together with \eqref{eq:Phi_fixed_point}, this verifies both the optimality and consistency conditions in the definition of a mean field Nash equilibrium.
\end{proof}

\subsection{Uniqueness of the Nash Equilibrium}
\label{subsec3.4}

We now establish uniqueness of the mean field Nash equilibrium. The argument applies the monotonicity conditions in Assumption \ref{asmp007} directly to the equilibrium FBSDE.

\begin{theorem}
\label{thm002}
Suppose that the assumptions of Theorem \ref{thrm001} and Assumption \ref{asmp007} hold. Then the mean field Nash equilibrium is unique. More precisely, the equilibrium measure flow is unique in $\mathcal M$, and the associated optimal control is unique up to $\mathrm dt\otimes\mathrm d\mathbb P\otimes\lambda$-null sets.
\end{theorem}

\begin{proof}
Let $\mu,\nu\in\mathcal M$ be the measure flows of two mean field Nash equilibria. By Proposition \ref{prop:frozen_optimizer}, their corresponding optimal controls are uniquely represented by the reduced FBSDEs. Denote the associated state-adjoint systems by
$$
\bigl\{ (X^\theta,p^\theta,q^\theta) \bigr\}_{\theta\in\Theta}
$$
and
$$
\bigl\{ (\widetilde X^\theta, \widetilde p^\theta, \widetilde q^\theta) \bigr\}_{\theta\in\Theta}.
$$
Both systems have the same prescribed initial random variable $\xi^\theta$ and may be realized with the same Brownian motion $W$. Thus
$$
X_0^\theta = \widetilde X_0^\theta = \xi^\theta, 
\qquad 
\lambda\text{-a.e. }\theta.
$$

Set
$$
\Delta X_t^\theta := X_t^\theta-\widetilde X_t^\theta,
\qquad
\Delta p_t^\theta := p_t^\theta-\widetilde p_t^\theta,
$$
and define
$$
D(t) := \int_\Theta \mathbb E \left[ \left\langle \Delta X_t^\theta,\Delta p_t^\theta \right\rangle \right] \lambda(\mathrm d\theta).
$$
Clearly, $D(0)=0$. 

Since $\mu$ and $\nu$ satisfy the consistency condition,
$$
m_t^K[\mu] = \mathfrak m[X_t],
\qquad
m_t^K[\nu] = \mathfrak m[\widetilde X_t].
$$
For brevity, write $m_t:=m_t^K[\mu]$, $\widetilde m_t:=m_t^K[\nu]$.
The two equilibrium systems satisfy
$$
\left\{
\begin{aligned}
\mathrm dX_t^\theta
& = B(t,X_t^\theta,p_t^\theta,m_t(\theta),\theta) \mathrm dt + \sigma\mathrm dW_t,
\\
\mathrm dp_t^\theta
& = -F(t,X_t^\theta,p_t^\theta,m_t(\theta),\theta) \mathrm dt + q_t^\theta\mathrm dW_t,
\end{aligned}
\right.
$$
and the analogous equations with $(X,p,q,m)$ replaced by $(\widetilde X,\widetilde p,\widetilde q,\widetilde m)$.

Define
$$
\begin{aligned}
\Delta B_t^\theta :={}
& B(t,X_t^\theta,p_t^\theta,m_t(\theta),\theta) - B(t,\widetilde X_t^\theta,\widetilde p_t^\theta, \widetilde m_t(\theta),\theta),
\\
\Delta F_t^\theta :={}
& F(t,X_t^\theta,p_t^\theta,m_t(\theta),\theta) - F(t,\widetilde X_t^\theta,\widetilde p_t^\theta, \widetilde m_t(\theta),\theta).
\end{aligned}
$$
Since the diffusion coefficient is the same constant matrix in both
forward equations, $\mathrm d\Delta X_t^\theta = \Delta B_t^\theta\,\mathrm dt$. Applying It\^o's formula to $\langle\Delta X_t^\theta,\Delta p_t^\theta\rangle$, taking expectations, and then integrating over $\Theta$, we obtain
$$
D(t) = \int_0^t \int_\Theta \mathbb E \left[ \left\langle \Delta B_s^\theta,\Delta p_s^\theta \right\rangle - \left\langle \Delta F_s^\theta,\Delta X_s^\theta \right\rangle \right] \lambda(\mathrm d\theta)\,\mathrm ds.
$$
Indeed, the stochastic integral has zero expectation, while the integrand on the right-hand side belongs to $L^1(0,T)$ by the Lipschitz and linear-growth bounds in Assumption~\ref{asmp006}(a), together with the square-integrability of the two equilibrium state-adjoint systems. Hence $D$ is absolutely continuous on $[0,T]$, and therefore
$$
D'(t) = \int_\Theta \mathbb E \left[ \left\langle \Delta B_t^\theta,\Delta p_t^\theta \right\rangle - \left\langle \Delta F_t^\theta,\Delta X_t^\theta \right\rangle \right] \lambda(\mathrm d\theta)
$$
for a.e. $t\in[0,T]$. By \eqref{eq:kernel_monotonicity},
\begin{equation}
\label{eq:Dprime_positive}
D'(t) \ge c_{\mathrm M} \int_\Theta \mathbb E|\Delta X_t^\theta|^2 \lambda(\mathrm d\theta) \ge 0.
\end{equation}
Hence $D$ is nondecreasing and, since $D(0)=0$, $D(T)\ge0$.

At the terminal time,
$$
p_T^\theta = \Gamma(X_T^\theta,m_T(\theta),\theta),
$$
and
$$
\widetilde p_T^\theta = \Gamma(\widetilde X_T^\theta, \widetilde m_T(\theta),\theta).
$$
Consequently,
$$
\begin{aligned} 
D(T) = \int_\Theta \mathbb E \Big[ \big\langle & \Gamma(X_T^\theta,m_T(\theta),\theta) - \Gamma(\widetilde X_T^\theta, \widetilde m_T(\theta),\theta),
\\
& X_T^\theta-\widetilde X_T^\theta \big\rangle \Big] \lambda(\mathrm d\theta).
\end{aligned}
$$
By \eqref{eq:kernel_terminal_monotonicity},
$$
D(T) \le -c_{\mathrm T} \int_\Theta \mathbb E |\Delta X_T^\theta|^2 \lambda(\mathrm d\theta) \le 0.
$$
Thus $D(T)=0$.

Since $D$ is nondecreasing and $D(0)=D(T)=0$, we have $D(t)=0$, $t\in[0,T]$. Hence $D'(t)=0$ for a.e. $t$. Combining this with
\eqref{eq:Dprime_positive} yields that for a.e. $t\in[0,T]$,
$$
\int_\Theta \mathbb E |\Delta X_t^\theta|^2 \lambda(\mathrm d\theta) = 0.
$$

The map
$$
t\longmapsto \int_\Theta \mathbb E |\Delta X_t^\theta|^2 \lambda(\mathrm d\theta)
$$
is continuous. Therefore the preceding equality holds for every $t\in[0,T]$. Using the joint realization as a coupling,
$$
\begin{aligned}
\int_\Theta \mathcal W_2 \left( \mu_t^\theta,\nu_t^\theta \right)^2 \lambda(\mathrm d\theta)
&\le \int_\Theta \mathbb E |\Delta X_t^\theta|^2 \lambda(\mathrm d\theta)
\\
&=0.
\end{aligned}
$$
Thus $d(\mu,\nu)=0$, and hence $\mu=\nu$ in $\mathcal M$.

Finally, once the equilibrium measure flow is fixed, Proposition \ref{prop:frozen_optimizer} gives a unique optimal control for $\lambda$-a.e. type $\theta$. Therefore the two equilibrium controls also coincide $\mathrm dt\otimes\mathrm d\mathbb P\otimes\lambda$-a.e. This proves uniqueness of the mean field Nash equilibrium.
\end{proof}

\begin{remark}
The smallness condition used for existence and the monotonicity condition used for uniqueness play different roles. In particular, the existence proof only requires the kernel interaction to be weak enough to preserve an invariant ball for the consistency map, whereas Assumption \ref{asmp007} requires the intrinsic monotonicity of the state-adjoint system to dominate the interaction also in the comparison of two self-consistent equilibria.

Alternatively, a smallness condition satisfying $C_\Phi\kappa_K<1$ in Proposition \ref{prop:Phi_stability} would make the consistency map a contraction and yield uniqueness directly. Assumption \ref{asmp007} is retained here because it permits uniqueness beyond that contraction regime.
\end{remark}

\section{An Explicit Kernel-Based Mean Field Game on a Continuum Type Space}
\label{sec4}

In this section, we consider a linear-quadratic instance of the kernel-based heterogeneous mean field game introduced in Section \ref{sec2}. In addition to heterogeneity in the interaction structure, we explicitly incorporate heterogeneity in individual preferences. More precisely, the type variable determines both how an agent is connected to the population through the interaction kernel and how strongly the agent penalizes state risk through a scalar risk-aversion profile.

The example serves two purposes. First, it provides a directly verifiable model satisfying the assumptions used in Sections \ref{sec2}--\ref{sec3}. The example yields a spectral smallness condition for existence and an explicit sufficient condition for uniqueness. Second, when the interaction kernel has finite rank, its eigenfunctions identify the finitely many collective patterns transmitted directly through the mean field. In contrast to the homogeneous-preference case, however, a nonconstant risk-aversion profile generally couples these kernel modes. This provides a simple illustration of the distinction between heterogeneity in interactions and heterogeneity in individual preferences.

\subsection{Model Formulation and Equilibrium Characterization}
\label{subsec4.1}

Let $\Theta=[0,1]$ be equipped with the Lebesgue probability measure, and let $U=\mathbb R^2$. For each $\theta\in[0,1]$, consider a two-dimensional state process $X_t^\theta\in\mathbb R^2$ and a two-dimensional control $u_t^\theta\in\mathbb R^2$.

The controlled state equation is
\begin{equation}
\label{eq:LQ_state}
\mathrm dX_t^\theta = \left( AX_t^\theta+Du_t^\theta+Cm_t(\theta) \right)\mathrm dt + \Sigma\,\mathrm dW_t,
\qquad
X_0^\theta=\xi^\theta,
\end{equation}
where $A,C,D,\Sigma\in\mathbb R^{2\times2}$ are constant matrices and $\Sigma$ is nondegenerate. We assume that $(\theta,\omega)\mapsto\xi^\theta(\omega)$ is jointly measurable, $\xi^\theta$ is $\mathcal F_0$-measurable for a.e. $\theta$ and
$$
\int_0^1 \mathbb E|\xi^\theta|^2\,\mathrm d\theta <\infty.
$$

The kernel-based mean field is
\begin{equation}
\label{eq:LQ_mean_field}
m_t(\theta) = \int_0^1 K(\theta,\theta')\bar X_t(\theta')\,\mathrm d\theta',
\qquad
\bar X_t(\theta) := \mathbb E[X_t^\theta].
\end{equation}
Here $K:[0,1]^2\to\mathbb R$ is a scalar kernel. In relation to the general matrix-valued framework of Section \ref{sec2}, we identify $K$ with $KI_2$ and let the associated integral operator act componentwise on $L^2([0,1];\mathbb R^2)$. Its operator norm is therefore equal to that of the scalar operator on $L^2([0,1])$.

We assume that $K\in L^2([0,1]^2)$ and $K(\theta,\theta')=K(\theta',\theta)$ a.e., and write
$$
\kappa_K := \|\mathcal K\|_{\mathcal L(L^2([0,1]))}.
$$
In accordance with Assumption \ref{asmp003}, we impose
\begin{equation}
\label{eq:LQ_kernel_smallness}
\kappa_K<\bar\kappa,
\end{equation}
where $\bar\kappa$ is the threshold appearing in \eqref{eq:kernel_smallness}.

To model heterogeneity in individual preferences, let $\gamma:[0,1]\to(0,\infty)$ be a measurable function satisfying
\begin{equation}
\label{eq:risk_aversion_bounds}
0<\underline\gamma
\le
\gamma(\theta)
\le
\overline\gamma
<\infty
\qquad
\text{for a.e. }\theta\in[0,1].
\end{equation}
We interpret $\gamma(\theta)$ as a reduced-form scalar risk-aversion intensity: larger values correspond to a stronger penalization of state risk. The use of a scalar parameter to order individual attitudes toward risk is consistent with the classical treatment of risk aversion through scalar measures of utility curvature; see, for example, \cite{pratt1978risk}. Here $\gamma(\theta)$ should not be identified literally with an Arrow-Pratt coefficient. Rather, it scales the quadratic state penalty in the present linear-quadratic formulation.

Let $Q_0,G_0\in\mathbb R^{2\times2}$ be symmetric positive-definite matrices and suppose that
\begin{equation}
\label{eq:LQ_base_coercivity}
Q_0\succeq q_0I_2,
\qquad
G_0\succeq g_0I_2
\end{equation}
for some $q_0,g_0>0$. We set
$$
Q(\theta):=\gamma(\theta)Q_0,
\qquad
G(\theta):=\gamma(\theta)G_0.
$$
Thus $Q_0$ and $G_0$ determine the common structure of the state penalty across the population, whereas $\gamma(\theta)$ determines how strongly a type-$\theta$ agent penalizes this risk.

The representative agent maximizes
\begin{equation}
\label{eq:LQ_reward}
\begin{aligned}
J^\theta(u^\theta) = \mathbb E\Bigg[ \int_0^T \Big(
& -\frac12 \gamma(\theta) (X_t^\theta)^\top Q_0X_t^\theta -\frac12 (u_t^\theta)^\top Ru_t^\theta
\\
& +(X_t^\theta)^\top Sm_t(\theta) \Big)\,\mathrm dt - \frac12 \gamma(\theta) (X_T^\theta)^\top G_0X_T^\theta \Bigg],
\end{aligned}
\end{equation}
where $R\succ0$ and $S\in\mathbb R^{2\times2}$ is a fixed interaction matrix. Notice that the type variable enters the model in two different ways. The kernel $K(\theta,\theta')$ determines the interaction pattern between types, whereas $\gamma(\theta)$ describes heterogeneity in individual preferences.

For a prescribed mean-field path $m$, the Hamiltonian is
$$
\begin{aligned}
H(x,u,p;\theta,m) ={}
& p^\top(Ax+Du+Cm) -\frac12\gamma(\theta)x^\top Q_0x -\frac12u^\top Ru +x^\top Sm.
\end{aligned}
$$
Since $R\succ0$, the mapping $u\mapsto H(x,u,p;\theta,m)$ is strictly concave and tends to $-\infty$ as $|u|\to\infty$. Its unique maximizer is determined by $D^\top p-Ru=0$, and hence
\begin{equation}
\label{eq:LQ_feedback}
\hat u(x,p,m,\theta) = R^{-1}D^\top p.
\end{equation}
In particular, the maximizer is jointly measurable and satisfies the linear-growth condition in Assumption \ref{asmp005}. Although the feedback map itself does not depend explicitly on $\theta$, the equilibrium adjoint process $p^\theta$ does, through the type-dependent state and terminal penalties.

Set
\begin{equation}
\label{eq:LQ_M}
M:=DR^{-1}D^\top.
\end{equation}
The reduced Hamiltonian is
\begin{equation}
\label{eq:LQ_reduced_H}
\begin{aligned}
\widehat H(x,p;\theta,m) ={}
& p^\top(Ax+Cm) +\frac12p^\top Mp -\frac12\gamma(\theta)x^\top Q_0x +x^\top Sm.
\end{aligned}
\end{equation}
Consequently,
$$
\partial_p\widehat H(x,p;\theta,m) = Ax+Mp+Cm,
$$
and
$$
\partial_x\widehat H(x,p;\theta,m) = A^\top p-\gamma(\theta)Q_0x+Sm.
$$
Thus the reduced coefficients introduced in Subsection \ref{subsec2.2} are
\begin{equation}
\label{eq:LQ_reduced_coefficients}
\begin{aligned}
B_{\rm red}(x,p,m,\theta)
&= Ax+Mp+Cm,
\\
F_{\rm red}(x,p,m,\theta)
&= A^\top p-\gamma(\theta)Q_0x+Sm,
\\
\Gamma(x,m,\theta)
&= -\gamma(\theta)G_0x.
\end{aligned}
\end{equation}

The boundedness of $\gamma$ in \eqref{eq:risk_aversion_bounds} implies that these coefficients satisfy the uniform Lipschitz and growth conditions in Assumption \ref{asmp006}(a). We next verify the monotonicity condition. For a fixed $m\in\mathbb R^2$,
$$
\begin{aligned}
& \left\langle B_{\rm red}(x,p,m,\theta) - B_{\rm red}(x',p',m,\theta), p-p' \right\rangle
\\
&\quad - \left\langle F_{\rm red}(x,p,m,\theta) - F_{\rm red}(x',p',m,\theta), x-x' \right\rangle
\\
&= (p-p')^\top M(p-p') + \gamma(\theta) (x-x')^\top Q_0(x-x').
\end{aligned}
$$
Since $M\succeq0$, \eqref{eq:risk_aversion_bounds} and \eqref{eq:LQ_base_coercivity} imply
$$
\begin{aligned}
& \left\langle B_{\rm red}(x,p,m,\theta) - B_{\rm red}(x',p',m,\theta), p-p' \right\rangle
\\
&\quad - \left\langle F_{\rm red}(x,p,m,\theta) - F_{\rm red}(x',p',m,\theta), x-x' \right\rangle
\\
&\ge \underline\gamma q_0|x-x'|^2.
\end{aligned}
$$
Moreover,
$$
\begin{aligned}
& \left\langle \Gamma(x,m,\theta)-\Gamma(x',m,\theta), x-x' \right\rangle
\\
&= -\gamma(\theta) (x-x')^\top G_0(x-x') \le -\underline\gamma g_0|x-x'|^2.
\end{aligned}
$$
Hence Assumption \ref{asmp006}(b) holds with, for instance,
$$
\beta=\underline\gamma q_0,
\qquad
\nu=\underline\gamma g_0.
$$

The equilibrium FBSDE therefore takes the form
\begin{equation}
\label{eq:LQ_equilibrium_FBSDE}
\left\{
\begin{aligned}
\mathrm dX_t^\theta ={}
& \left( AX_t^\theta+Mp_t^\theta+Cm_t(\theta) \right)\mathrm dt + \Sigma\,\mathrm dW_t,
\\
\mathrm dp_t^\theta ={}
& - \left( A^\top p_t^\theta - \gamma(\theta)Q_0X_t^\theta + Sm_t(\theta) \right)\mathrm dt + q_t^\theta\,\mathrm dW_t,
\\
X_0^\theta ={}
& \xi^\theta,
\\
p_T^\theta ={}
& -\gamma(\theta)G_0X_T^\theta,
\\
m_t(\theta) ={}
& \displaystyle \int_0^1 K(\theta,\theta') \mathbb E[X_t^{\theta'}]\,\mathrm d\theta'.
\end{aligned}
\right.
\end{equation}

The preceding conditions verify Assumptions \ref{asmp001}--\ref{asmp006}. In particular, if \eqref{eq:LQ_kernel_smallness} holds, Theorem \ref{thrm001} yields the existence of a mean field Nash equilibrium.

We next give an explicit sufficient condition for uniqueness. Suppose in addition that
\begin{equation}
\label{eq:LQ_M_coercive}
M\succeq r_0I_2
\end{equation}
for some $r_0>0$. This condition holds, for example, when $D$ is nonsingular.

Let $(X,p)$ and $(\widetilde X,\widetilde p)$ be two type-indexed square-integrable state-adjoint families, write
$$
\Delta X^\theta = X^\theta-\widetilde X^\theta,
\qquad
\Delta p^\theta = p^\theta-\widetilde p^\theta,
$$
$$
\begin{aligned}
\Delta B_{\mathrm{red}}^\theta
&:= B_{\mathrm{red}}(X^\theta,p^\theta,m[X](\theta),\theta) -B_{\mathrm{red}}(\widetilde X^\theta,\widetilde p^\theta,m[\widetilde X](\theta),\theta),
\\
\Delta F_{\mathrm{red}}^\theta
&:= F_{\mathrm{red}}(X^\theta,p^\theta,m[X](\theta),\theta) -F_{\mathrm{red}}(\widetilde X^\theta,\widetilde p^\theta,m[\widetilde X](\theta),\theta).
\end{aligned}
$$
and
$$
\Delta m = \mathfrak m[X]-\mathfrak m[\widetilde X].
$$
Then
$$
\begin{aligned}
& \left\langle \Delta B_{\rm red}^\theta,\Delta p^\theta \right\rangle - \left\langle \Delta F_{\rm red}^\theta,\Delta X^\theta \right\rangle
\\
&= (\Delta p^\theta)^\top M\Delta p^\theta + \gamma(\theta) (\Delta X^\theta)^\top Q_0\Delta X^\theta
\\
&\quad + \left\langle C\Delta m(\theta),\Delta p^\theta \right\rangle - \left\langle S\Delta m(\theta),\Delta X^\theta \right\rangle.
\end{aligned}
$$
Furthermore,
$$
\|\Delta m\|_{L^2} \le \kappa_K \left( \int_0^1 \mathbb E|\Delta X^\theta|^2\,\mathrm d\theta \right)^{1/2}.
$$
Using \eqref{eq:LQ_M_coercive}, Jensen's inequality, and Young's inequality, we obtain
$$
\begin{aligned}
& \int_0^1 \mathbb E \left[ \left\langle \Delta B_{\rm red}^\theta,\Delta p^\theta \right\rangle - \left\langle \Delta F_{\rm red}^\theta,\Delta X^\theta \right\rangle \right] \mathrm d\theta
\\
&\ge \left( \underline\gamma q_0 - \|S\|\kappa_K - \frac{\|C\|^2\kappa_K^2}{4r_0} \right) \int_0^1 \mathbb E|\Delta X^\theta|^2\,\mathrm d\theta.
\end{aligned}
$$
Hence Assumption \ref{asmp007}(a) holds whenever
\begin{equation}
\label{eq:LQ_uniqueness_condition}
\underline\gamma q_0 > \|S\|\kappa_K + \frac{\|C\|^2\kappa_K^2}{4r_0}.
\end{equation}

For the terminal condition,
$$
\begin{aligned}
& \int_0^1 \mathbb E \left[ \left\langle -\gamma(\theta)G_0X^\theta + \gamma(\theta)G_0\widetilde X^\theta, X^\theta-\widetilde X^\theta \right\rangle \right] \mathrm d\theta
\\
&\le -\underline\gamma g_0 \int_0^1 \mathbb E |X^\theta-\widetilde X^\theta|^2 \,\mathrm d\theta.
\end{aligned}
$$
Therefore Assumption \ref{asmp007}(b) holds with
$$
c_{\mathrm T} = \underline\gamma g_0.
$$

We have thus obtained the following direct consequence of the general theory.

\begin{proposition}
\label{prop:LQ_existence_uniqueness}
Suppose that \eqref{eq:risk_aversion_bounds} and \eqref{eq:LQ_base_coercivity} hold, $R\succ0$, and \eqref{eq:LQ_kernel_smallness} is satisfied. Then the linear-quadratic game \eqref{eq:LQ_state}--\eqref{eq:LQ_reward} admits at least one mean field Nash equilibrium.

If, in addition, \eqref{eq:LQ_M_coercive} and \eqref{eq:LQ_uniqueness_condition} hold, then the equilibrium is unique.
\end{proposition}

\subsection{Kernel Modes and Preference-Induced Mode Mixing}
\label{subsec4.2}

We next examine how the interaction kernel and the heterogeneous risk-aversion profile jointly determine the equilibrium mean dynamics.

Suppose that $K$ has finite rank and admits the representation
\begin{equation}
\label{eq:finite_rank_kernel}
K(\theta,\theta') = \sum_{n=0}^N \rho_ne_n(\theta)e_n(\theta'),
\end{equation}
where $\{e_n\}_{n=0}^N$ is an orthonormal family in $L^2([0,1])$ and $\rho_0,\ldots,\rho_N\in\mathbb R$. Then
$$
\mathcal Ke_n = \rho_ne_n,
\qquad
n=0,\ldots,N,
$$
and
\begin{equation}
\label{eq:finite_rank_operator_norm}
\kappa_K = \max_{0\le n\le N}|\rho_n|.
\end{equation}
Thus the kernel smallness condition becomes the directly verifiable spectral condition
\begin{equation}
\label{eq:spectral_existence_condition}
\max_{0\le n\le N}|\rho_n| < \bar\kappa.
\end{equation}

Extend $\{e_0,\ldots,e_N\}$ to a complete orthonormal basis $\{e_n\}_{n\ge0}$ of $L^2([0,1])$, and set $\rho_n:=0$, $n>N$. Define the equilibrium mean profiles
$$
\bar X_t(\theta) := \mathbb E[X_t^\theta],
\qquad
\bar p_t(\theta) := \mathbb E[p_t^\theta].
$$
Taking expectations in \eqref{eq:LQ_equilibrium_FBSDE} gives
\begin{equation}
\label{eq:LQ_mean_system}
\left\{
\begin{aligned}
\dot{\bar X}_t(\theta)
&= A\bar X_t(\theta) + M\bar p_t(\theta) + C\mathcal K\bar X_t(\theta),
\\
\dot{\bar p}_t(\theta)
&= -A^\top\bar p_t(\theta) + \gamma(\theta)Q_0\bar X_t(\theta) - S\mathcal K\bar X_t(\theta),
\\
\bar p_T(\theta)
&= -\gamma(\theta)G_0\bar X_T(\theta).
\end{aligned}
\right.
\end{equation}

For $n\ge0$, define
$$
x_n(t) := \int_0^1 \bar X_t(\theta)e_n(\theta)\,\mathrm d\theta,
$$
$$
p_n(t) := \int_0^1 \bar p_t(\theta)e_n(\theta)\,\mathrm d\theta,
$$
and
$$
x_{n,0} := \int_0^1 \mathbb E[\xi^\theta]e_n(\theta)\,\mathrm d\theta.
$$
Then
$$
\mathcal K\bar X_t(\theta) = \sum_{n=0}^N \rho_nx_n(t)e_n(\theta).
$$
Thus only the first $N$ modes enter the mean-field interaction directly.

The nonconstant risk-aversion profile introduces a second operator on the type space, namely the multiplication operator
$$
(\mathcal Gv)(\theta) := \gamma(\theta)v(\theta).
$$
Its matrix coefficients in the basis $\{e_n\}_{n\ge0}$ are
\begin{equation}
\label{eq:risk_mixing_coefficients}
a_{nj} := \int_0^1 \gamma(\theta)e_n(\theta)e_j(\theta)\,\mathrm d\theta,
\qquad
n,j\ge0.
\end{equation}
Since $\gamma$ is bounded, $\mathcal G$ is a bounded self-adjoint operator on $L^2([0,1])$, and the matrix $(a_{nj})_{n,j\ge0}$ represents this operator in the kernel eigenbasis.

Projecting \eqref{eq:LQ_mean_system} onto $e_n$ gives, for every $n\ge0$,
\begin{equation}
\label{eq:modal_fbsde}
\left\{
\begin{aligned}
\dot x_n(t)
&= (A+\rho_nC)x_n(t) + Mp_n(t),
\\
\dot p_n(t)
&= -A^\top p_n(t) + Q_0\sum_{j\ge0}a_{nj}x_j(t) - \rho_nSx_n(t),
\\
x_n(0)
&= x_{n,0},
\\
p_n(T)
&= -G_0\sum_{j\ge0}a_{nj}x_j(T).
\end{aligned}
\right.
\end{equation}
Here and below $\rho_n=0$ for $n>N$.

Equation \eqref{eq:modal_fbsde} separates the two sources of heterogeneity. The eigenvalues $\rho_n$ determine which collective patterns are transmitted through the population and with what strength, while the coefficients $a_{nj}$ describe how heterogeneous risk preferences mix these patterns in the adjoint dynamics. In particular, finite rank of $\mathcal K$ means that only finitely many modes enter the interaction term directly, but it does not in general make the complete mean system finite dimensional.

This distinction disappears when risk preferences are homogeneous.
Indeed, if $\gamma(\theta)\equiv\gamma_0$, then $a_{nj}=\gamma_0\delta_{nj}$, and \eqref{eq:modal_fbsde} decouples into independent modal systems:
$$
\left\{
\begin{aligned}
\dot x_n(t)
&= (A+\rho_nC)x_n(t)+Mp_n(t),
\\
\dot p_n(t)
&= -A^\top p_n(t) + (\gamma_0Q_0-\rho_nS)x_n(t),
\\
x_n(0)
&= x_{n,0},
\\
p_n(T)
&= -\gamma_0G_0x_n(T).
\end{aligned}
\right.
$$
Thus the independent modal decomposition obtained under homogeneous preferences is a special case. For a genuinely heterogeneous $\gamma$, the multiplication operator $\mathcal G$ generally does not commute with $\mathcal K$, and the off-diagonal coefficients $a_{nj}$ transmit information between different kernel modes.

The same representation gives a spectral form of the sufficient uniqueness condition. Let
$$
\rho_{\max} := \max_{0\le n\le N}|\rho_n|.
$$
Then \eqref{eq:LQ_uniqueness_condition} becomes
\begin{equation}
\label{eq:spectral_uniqueness_condition}
\underline\gamma q_0 > \|S\|\rho_{\max} + \frac{\|C\|^2\rho_{\max}^2}{4r_0}.
\end{equation}
Hence uniqueness is ensured when the uniform lower bound on the individual state penalty dominates the effective strength of the kernel-mediated interaction.

The coupled representation \eqref{eq:modal_fbsde} also suggests a natural finite-dimensional approximation. Truncating the complete basis after $L\ge N$ modes replaces the multiplication operator $\mathcal G$ by the matrix
$$
A_L^\gamma := (a_{nj})_{0\le n,j\le L}
$$
and yields a deterministic forward-backward system with modal state and adjoint vectors in $\mathbb R^{2L}$. Such a truncation preserves all $N$ directly interacting kernel modes while successively incorporating the additional mode mixing generated by heterogeneous preferences. We do not pursue approximation error estimates here.

\subsection{Financial Interpretation}
\label{subsec4.3}

We finally illustrate the preceding structure through a heterogeneous financial market interpretation.

Let the type parameter $\theta\in[0,1]$ index an investor's risk tolerance, with smaller values corresponding to more conservative investors and larger values to more aggressive investors. Consistently with this ordering, choose the decreasing risk-aversion profile
\begin{equation}
\label{eq:financial_risk_profile}
\gamma(\theta) = \gamma_{\max} - (\gamma_{\max}-\gamma_{\min})\theta,
\qquad
0<\gamma_{\min}<\gamma_{\max}.
\end{equation}
Thus
$$
\underline\gamma=\gamma_{\min},
\qquad
\overline\gamma=\gamma_{\max}.
$$
The type parameter therefore has a direct preference interpretation: investors near $\theta=0$ place a larger weight on the quadratic state penalty, whereas investors near $\theta=1$ tolerate a larger exposure to the same underlying state risk.

Write
$$
X_t^\theta =
\begin{pmatrix}
S_t^\theta
\\
H_t^\theta
\end{pmatrix},
$$
where $S_t^\theta$ represents the investor's market sentiment or subjective valuation of a representative asset, while $H_t^\theta$ denotes the corresponding position size. The control $u_t^\theta$ describes the rate at which the investor adjusts these state variables. The matrices $Q_0$ and $G_0$ determine the relative penalties attached to deviations in the two state components, while $\gamma(\theta)$ scales their overall importance for the investor.

This preference heterogeneity is distinct from the interaction heterogeneity generated by the kernel. For example, the component of the mean field corresponding to market sentiment contains a term of the form
$$
\int_0^1 K(\theta,\theta') \mathbb E[S_t^{\theta'}] \,\mathrm d\theta'.
$$
Thus $\gamma(\theta)$ describes how strongly an investor evaluates state risk, whereas $K(\theta,\theta')$ describes how the aggregate behavior of type $\theta'$ affects type $\theta$.

To obtain an explicit three-mode interaction structure, define
$$
e_0(\theta)=1,
$$
$$
e_1(\theta) = \sqrt2\cos(\pi\theta),
$$
and
\begin{equation}
\label{eq:financial_e2}
e_2(\theta) = \frac{ \sin(\pi\theta)-2/\pi }{ \sqrt{1/2-4/\pi^2} }.
\end{equation}
Then $\{e_0,e_1,e_2\}$ is an orthonormal family in
$L^2([0,1])$.

Consider the kernel of rank at most three
\begin{equation}
\label{eq:financial_kernel}
K(\theta,\theta')
={}
\kappa_0e_0(\theta)e_0(\theta')
+
\kappa_1e_1(\theta)e_1(\theta')
+
\kappa_2e_2(\theta)e_2(\theta'),
\end{equation}
where $\kappa_0,\kappa_1,\kappa_2\in\mathbb R$. The corresponding eigenvalues are
$$
\rho_0=\kappa_0,
\qquad
\rho_1=\kappa_1,
\qquad
\rho_2=\kappa_2,
$$
and hence $\kappa_K = \max_{0\le i\le2}|\kappa_i|$. The existence condition in Proposition \ref{prop:LQ_existence_uniqueness} therefore becomes
\begin{equation}
\label{eq:financial_existence_condition}
\max_{0\le i\le2}|\kappa_i| < \bar\kappa.
\end{equation}

The first mode $e_0(\theta)=1$ represents a market-wide consensus effect. It assigns the same loading to every investor type and therefore captures a common movement shared by the whole population.

The second mode $e_1(\theta) = \sqrt2\cos(\pi\theta)$ distinguishes the two ends of the risk-tolerance spectrum. It is positive for relatively conservative investors and negative for relatively aggressive investors. This mode therefore represents a polarization pattern between the two extremes of the type space.

The third mode $e_2$ 1n \eqref{eq:financial_e2} is positive around $\theta=1/2$ and negative near both endpoints. It contrasts investors with intermediate risk tolerance against investors located near the conservative and aggressive extremes. In this sense, it represents an intermediate-versus-extreme interaction pattern while remaining orthogonal to both the consensus and polarization modes.

The equilibrium mean profile may be written as
$$
\bar X_t(\theta) = x_0(t)e_0(\theta) + x_1(t)e_1(\theta) + x_2(t)e_2(\theta) + \bar X_t^\perp(\theta),
$$
where $\bar X_t^\perp\perp\operatorname{span}\{e_0,e_1,e_2\}$, and the interaction term is exactly
\begin{equation}
\label{eq:financial_mean_field_modes}
m_t(\theta) = \kappa_0x_0(t)e_0(\theta) + \kappa_1x_1(t)e_1(\theta) + \kappa_2x_2(t)e_2(\theta).
\end{equation}
Thus only these three collective patterns are transmitted directly through the mean field.

The heterogeneous preference profile \eqref{eq:financial_risk_profile}, however, mixes the modal state components in the adjoint equation. To show this explicitly, consider
$$
\Delta\gamma := \gamma_{\max}-\gamma_{\min}, \qquad \gamma_\text{avg} := \frac{\gamma_{\max}+\gamma_{\min}}{2}.
$$
For $i,j\in\{0,1,2\}$, define
$$
a_{ij} = \int_0^1 \gamma(\theta)e_i(\theta)e_j(\theta) \,\mathrm d\theta.
$$
A direct calculation gives
\begin{equation}
\label{eq:financial_preference_matrix}
\left(a_{ij}\right)_{0\le i,j\le2} =
\begin{pmatrix}
\gamma_\text{avg}
&
\dfrac{2\sqrt2\,\Delta\gamma}{\pi^2}
&
0
\\[3mm]
\dfrac{2\sqrt2\,\Delta\gamma}{\pi^2}
&
\gamma_\text{avg}
&
-\dfrac{(16-\pi^2)\Delta\gamma}
{2\pi^2\sqrt{\pi^2-8}}
\\[3mm]
0
&
-\dfrac{(16-\pi^2)\Delta\gamma}
{2\pi^2\sqrt{\pi^2-8}}
&
\gamma_\text{avg}
\end{pmatrix}.
\end{equation}
In particular, whenever $\gamma_{\max}\ne\gamma_{\min}$,  $a_{01}$ and $a_{12}$ are generally nonzero. Hence the polarization mode is coupled through individual preferences to both the consensus mode and the intermediate-versus-extreme mode. This coupling is absent when all investors have the same risk aversion, in which case $\Delta\gamma=0$ and the matrix in \eqref{eq:financial_preference_matrix} reduces to $\gamma_\text{avg} I_3$.

Equation \eqref{eq:financial_preference_matrix} describes only the projection of the multiplication operator onto the three directly interacting modes. Since multiplication by a nonconstant $\gamma(\theta)$ need not preserve $\operatorname{span}\{e_0,e_1,e_2\}$, the preference profile may also couple these modes to the orthogonal component $\bar X^\perp$. Thus the three-mode kernel determines a channel supported on three type modes of direct population interaction, while heterogeneous preferences can generate additional indirect modal coupling through the individual optimization problem.

The coefficients $\kappa_0,\kappa_1,\kappa_2$ and the profile $\gamma$ therefore play analytically distinct roles. The former determine the strength of the collective interaction patterns and satisfy
$$
\kappa_K = \max_{0\le i\le2}|\kappa_i|,
$$
whereas the latter determines the individual sensitivity to state risk and the degree of preference-induced mode mixing. In particular, under \eqref{eq:LQ_M_coercive}, a sufficient condition for uniqueness in this three-mode example is
\begin{equation}
\label{eq:financial_uniqueness_condition}
\gamma_{\min}q_0 > \|S\| \max_{0\le i\le2}|\kappa_i| + \frac{\|C\|^2}{4r_0} \left( \max_{0\le i\le2}|\kappa_i| \right)^2.
\end{equation}
This condition makes explicit the balance between individual risk penalization and collective interaction. A smaller $\gamma_{\min}$ permits some investor types to tolerate greater state risk and consequently weakens the uniform coercivity available for the uniqueness argument, whereas stronger kernel-mediated interaction increases the amount of coercivity required.

The example therefore separates two conceptually different forms of heterogeneity. The kernel specifies how different investor types are connected through the population, while the scalar profile $\gamma(\theta)$ specifies how their individual attitudes toward risk differ. Their interaction is reflected in the coupled modal system \eqref{eq:modal_fbsde}: kernel eigenvalues determine the directly transmitted collective modes, while heterogeneous preferences mix these modes through the multiplication operator generated by $\gamma$.

\section{Conclusion}\label{sec5}

In this paper, we developed a probabilistic existence and uniqueness theory for a class of heterogeneous mean field games on a continuum type space with kernel-based interactions. The interaction is represented by a Hilbert-Schmidt integral operator acting on the type-indexed mean state profile, while the individual control problem allows a general closed convex control set and a nonlinear feedback characterized by the unique maximizer of the Hamiltonian.

For a prescribed population flow, the stochastic maximum principle leads to a type-indexed family of fully coupled FBSDEs. A Peng-Wu type monotonicity condition on the reduced coefficients gives global well-posedness of this frozen system and, at the same time, provides the strong concavity needed to verify optimality of the resulting feedback. The equilibrium consistency problem can consequently be formulated as a single-valued fixed-point problem. Working with the kernel mean-field path, we used stability estimates for the frozen FBSDE, compactness of the Hilbert-Schmidt kernel operator, and a uniform time-modulus estimate to establish continuity and compactness of the consistency map. When the operator norm of the kernel is sufficiently small to preserve an invariant ball, Schauder's fixed-point theorem yields existence of a mean field Nash equilibrium.

Uniqueness was obtained from a different structural mechanism. We introduced an aggregate monotonicity condition in which the kernel mean fields are generated endogenously by the competing state profiles. An energy argument for the difference of two equilibrium FBSDEs then forces the corresponding state distributions, and subsequently the optimal controls, to coincide. This condition need not reduce to the contraction requirement $ C_\Phi\kappa_K<1 $; it therefore separates the uniqueness mechanism from the smallness condition used to construct the Schauder invariant set.

The linear-quadratic example illustrates how these abstract conditions can be checked. For finite-rank kernels, the directly interacting component is generated by finitely many kernel modes, while heterogeneous preferences induce additional coupling between these modes through a multiplication operator on the type space. The eigenvalues of the interaction operator determine both the strength of these modes and the quantitative existence and uniqueness conditions. The example therefore illustrates how a continuum of heterogeneous types may generate a low-dimensional collective interaction structure.

Several extensions appear natural. One important direction is to relax the strict-concavity and unique-maximizer assumptions on the Hamiltonian. If the individual best response is no longer unique, the present consistency map may cease to be single-valued. A natural approach is to enlarge the admissible strategies to relaxed controls, interpreted as mixed strategies, so that compactness and convexity can be recovered in a relaxed-control formulation. The equilibrium operator would then become a set-valued correspondence. An upper hemicontinuous correspondence with nonempty compact convex values on a suitable compact convex invariant domain would allow the use of the Kakutani-Fan-Glicksberg fixed-point theorem, in the spirit of the relaxed-control approach to general mean field games developed in \cite{lacker2015mfg}. Under additional convexity or purification assumptions, one may then ask whether a relaxed equilibrium can be represented by a strict control.

A second direction is to weaken or remove the smallness restriction on $\kappa_K$. The present condition enters only through the invariant-ball estimate, rather than through well-posedness of the frozen FBSDE or continuity of the response map. Sharper dissipativity estimates, continuation arguments, or alternative fixed-point principles may therefore permit stronger interactions. It would also be useful to extend the model to time-dependent or random interaction kernels and to study how perturbations of the kernel affect the equilibrium.

Another natural extension is to replace the first-moment statistic used in the present kernel aggregation by a more general functional of the type-dependent distribution. More precisely, instead of $a_t^\mu(\theta)=\int_{\mathbb R^k}x\,\mu_t^\theta(\mathrm dx)$, one may consider $a_t^\mu(\theta) = \mathcal A(\theta,\mu_t^\theta)$ for a suitable measure functional $\mathcal A$, and define
$$
m_t^{K,\mathcal A}[\mu](\theta) = \int_\Theta K(\theta,\theta') \mathcal A(\theta',\mu_t^{\theta'}) \,\lambda(\mathrm d\theta').
$$
This formulation would separate the statistic extracted from each type-dependent distribution from its aggregation across the type space, and could accommodate higher-order moments or nonlinear distributional features. Higher-order moments may require stronger moment assumptions and a different functional setting. Under suitable Wasserstein continuity and growth conditions on $\mathcal A$, part of the stability and compactness mechanism developed here is expected to persist, whereas the equilibrium monotonicity conditions and the corresponding uniqueness argument would need to be reformulated for the new statistic; under an analogous condition, the energy argument for uniqueness could still be retained.

Finally, the finite-rank analysis suggests a systematic approximation theory for general kernels. Finite-rank approximations of the interaction operator, for example by singular-value truncation, provide a natural starting point for equilibrium approximation. With heterogeneous preferences, finite rank of the interaction operator alone does not make the complete mean system finite dimensional. Quantifying the resulting approximation error would be useful both analytically and computationally. A complementary direction is to relate the present probabilistic formulation to master-equation or infinite-dimensional dynamic-programming approaches for heterogeneous kernel-based mean field games.

\backmatter

\section*{Declarations}
This work is supported by National Natural Science Foundation (NNSF) of China under Grant 2022YFA1006100.


\begin{appendices}




\end{appendices}


\bibliography{bibliography0616.bib}% common bib file

\end{document}